\documentclass[10pt]{article}
\usepackage{indentfirst,latexsym,bm}
\usepackage{amsfonts}
\usepackage{amssymb}
\usepackage{amsmath}
\usepackage{amsbsy}
\usepackage{dsfont}
\usepackage{amsthm}
\usepackage{amscd}
\usepackage[all]{xy}
\usepackage{chemarrow}
\usepackage{multirow}
\begin{document}

\title{Eight classes of new Hopf algebras of dimension $128$ without the Chevalley property}
\author{Naihong Hu\thanks{Email:\,nhhu@math.ecnu.edu.cn}\\{\small Department of Mathematics, SKLPMMP,} \\ {\small East China Normal University,
Shanghai 200062, China}
\and Rongchuan Xiong\thanks{Email:\,52150601006@ecnu.cn}
\\{\small Department of Mathematics, SKLPMMP,} \\ {\small East China Normal University,
Shanghai 200062, China}
}
\maketitle

\newtheorem{question}{Question}
\newtheorem{defi}{Definition}[section]
\newtheorem{conj}{Conjecture}
\newtheorem{thm}[defi]{Theorem}
\newtheorem{lem}[defi]{Lemma}
\newtheorem{pro}[defi]{Proposition}
\newtheorem{cor}[defi]{Corollary}
\newtheorem{rmk}[defi]{Remark}
\newtheorem{Example}{Example}[section]

\theoremstyle{plain}
\newcounter{maint}
\renewcommand{\themaint}{\Alph{maint}}
\newtheorem{mainthm}[maint]{Theorem}

\theoremstyle{plain}
\newtheorem*{proofthma}{Proof of Theorem A}
\newtheorem*{proofthmb}{Proof of Theorem B}

\newcommand{\K}{\mathds{k}}
\newcommand{\A}{\mathcal{A}}
\newcommand{\M}{\mathcal{M}}
\newcommand{\E}{\mathcal{E}}
\newcommand{\D}{\mathcal{D}}
\newcommand{\BN}{\mathcal{B}}
\newcommand{\Lam}{\lambda}
\newcommand{\HYD}{{}^{H}_{H}\mathcal{YD}}
\newcommand{\As}{^\ast}
\newcommand{\N}{\mathds{N}}
\newcommand{\Pp}{\mathcal{P}}
\newcommand{\LA}{\Lambda^5(\mu)}
\newcommand{\LAA}{\Lambda^6(\mu)}
\begin{abstract}
 Classifying Hopf algebras of a given dimension is a hard and open question. Using the generalized lifting method, we determine all finite-dimensional Hopf algebras over an algebraically closed field of characteristic zero whose coradical generates a Hopf algebra $H$ of dimension $16$ without the Chevalley property and the corresponding infinitesimal braidings are simple objects in $\HYD$. In particular, we figure out $8$ classes of new Hopf algebras of dimension $128$ without the Chevalley property.

\end{abstract}
\section{Introduction}
Let $\K$ be an algebraically closed field of characteristic zero. The question of classification of all Hopf algebras over $\K$ of a given dimension up to isomorphism was posed by I. Kaplansky in 1975. He conjectured that every Hopf algebra over $\K$ of prime dimension must be isomorphic to a group algebra which was proved by Y. Zhu \cite{Z94} in 1994. Since then, more and more mathematicians have been trying to classify finite-dimensional Hopf algebras of a given dimension and have made some progress. As the aforementioned, the classification of Hopf algebras of prime dimension $p$ has been completed by Y. Zhu \cite{Z94} and all of them are isomorphic to the cyclic group algebra of dimension $p$. Further results have completed the classification of Hopf algebras of dimension $p^2$ for $p$ a prime (see
\cite{Ma96,AS,Ng02}), of dimension $2p$ for $p$ an odd prime (see \cite{Ma95,Ng05}), and of
dimension $2p^2$ for $p$ an odd prime (see \cite{HN09}). For survey of the classification for dimensions up to $100$, we refer to \cite{BG13} and the references therein. Even though the classification of finite-dimensional pointed Hopf algebras as a special class (with abelian groups  as the coradicals) has made some astonishing breakthrough under the assumption that the order of the finite abelian groups are prime to $210$
(\cite{AS10}, etc.) in terms of the lifting method (only as one kind approach of the so-called principle realization) introduced by Andruskiewitsch-Schneider in \cite{AS98}, there are few unequivocal results and solving methods in the classification of Hopf algebras of a given dimension.  This constitutes actually a great challenge after Heckenberger et al's studies on the level of Nichols algebras (of group type).

 Comparison with other newly designed approaches, we would prefer to the lifting method since any new method inevitably encounters the common hard point question: how to derive the explicit relations involving generators. Let us briefly recall it. Let $A$ be a finite dimensional Hopf algebra such that the coradical $A_0$ is a Hopf subalgebra, which implies that the coradical filtration $\{A_n\}_{n=0}^{\infty}$ is a Hopf algebra filtration where $A_n=A_0\bigwedge A_{n-1}$. Let $G=\text{gr}\, A$ be the associated graded Hopf algebra, that is,
$G=\oplus_{n=0}^{\infty}G(n)$ with $G(0)=A_0$ and $G(n)=A_n/A_{n-1}$. Denote by $\pi: G\rightarrow A_0$ the Hopf
algebra projection of $G$ onto $A_0=G(0)$, then $\pi$ splits the inclusion $i: A_0\hookrightarrow G$ and thus by a theorem of Radford \cite{R85}, $G\cong R\sharp A_0$, where $R=G^{co\pi}=\{h\in H\mid (id\otimes \pi)\Delta_G(h)=h\otimes 1\}$ is a braided Hopf algebra in ${}^{A_0}_{A_0}\mathcal{YD}$. Moreover, $R=\oplus_{n=0}^{\infty}R(n)=R\cap G(n)$ with $R(0)=\K$ and $R(1)=\Pp(R)$, the space of primitive elements of $R$, which is a braided vector space called the infinitesimal braiding. In particular, the subalgebra generated as a braided Hopf algebra by $V$ is so-called Nichols algebra over $V$ denoted by $\BN(V)$, which plays a key role in the classification of pointed Hopf algebra under the  following
\begin{conj} $($\text{\rm Conjecture\,2.7 \cite{AS02}}$)$
Any finite-dimensional braided Hopf algebra $R$ in ${}^{A_0}_{A_0}\mathcal{YD}$
satisfying $\Pp(R)=R(1)$ is generated by $R(1)$.
\end{conj}
Usually, if we fix a Hopf algebra $H$, then the main steps of the lifting method for classification of all finite-dimensional Hopf algebras $A$ such that $A_0\cong H$ are:
\begin{itemize}
  \item Determine all objects $V$ in $\HYD$ such that Nichols algebras $\BN(V)$ are finite-dimensional and
  describe $\BN(V)$ explicitly in terms of generators and relations.
  \item For such $V$, determine all possible finite-dimensional Hopf algebras $A$ such that the associated graded
  Hopf algebras $\text{gr}\,A\cong \BN(V)\sharp H$. We call $A$ a lifting of $V$ $($or $\BN(V)$$)$ over $H$.
  \item Prove that any finite-dimensional Hopf algebra over $H$ is generated by the first term of the coradical filtration.
\end{itemize}
So far, the lifting method has produced many striking results of the classification of pointed or copointed Hopf algebras. For more details about the results, we refer to \cite{A14,BG13} and the references therein.

If $A$ is a Hopf algebra without the Chevalley property, then the coradical filtration $\{A_n\}_{n=0}^{\infty}$ is not a Hopf algebra filtration such that the associated graded coalgebra is no longer  a Hopf algebra. To overcome this obstacle, Andruskiewitsch and Cuadra \cite{AC13} extended the lifting method by replacing the coradical filtration $\{A_n\}_{n=0}^{\infty}$ by the standard filtration $\{A_{[n]}\}_{n=0}^{\infty}$, which is defined recursively by
\begin{itemize}
  \item $A_{[0]}$ to be the subalgebra generated by the coradical $A_0$;
  \item $A_{[n]}=A_{[n-1]}\bigwedge A_{[0]}$.
\end{itemize}
Especially, if $A_0$ is a Hopf algebra, then $A_{[0]}=A_0$ and standard filtration is just the coradical filtration. Under the assumption that $S_A(A_{[0]})\subseteq A_{[0]}$, it turns out that the standard filtration is a Hopf algebra filtration, and the associated graded coalgebra $\text{gr}\,A=\oplus_{n=0}^{\infty}A_{[n]}/A_{[n-1]}$ with $A_{[-1]}=0$
is also a Hopf algebra. If we denote as before, $G=\text{gr}\,A$ and $\pi:G\rightarrow A_0$ splits the inclusion $i:A_0\hookrightarrow G$. Thus by a theorem of Radford, $G\cong R\sharp A_0$, where $R=G^{co\pi}=\{h\in G\mid(id\otimes \pi)\Delta_G(h)=h\otimes 1\}$ is a braided Hopf algebra in ${}^{A_0}_{A_0}\mathcal{YD}$. Moreover,
$R=\oplus_{n=0}^{\infty}R(n)=R\cap G(n)$ with $R(0)=\K$ and $R(1)=\Pp(R)$, which is also a braided vector space called the infinitesimal braiding. This is summarized in the following
\begin{thm} $($\text{\rm Theorem\,1.3 \cite{AC13}}$)$
Any Hopf algebra with injective antipode is a deformation of the bosonization of a Hopf algebra generated by a
cosemisimple coalgebra by a connected graded Hopf algebra in the category of Yetter-Drinfeld modules over the latter.
\end{thm}

In order to produce new Hopf algebras by using the generalized lifting method, one needs to consider the following questions:
\begin{itemize}
  \item {\text{\rm Question\,\uppercase\expandafter{\romannumeral1} (\cite{AC13})}} Let $C$ be a cosemisimple coalgebra and $\mathcal{S} : C \rightarrow C$
  an injective anti-coalgebra morphism. Classify all Hopf algebras $H$ generated by $C$ such that $S|_C = \mathcal{S}$.
  \item {\text{\rm Question\,\uppercase\expandafter{\romannumeral2} (\cite{AC13})}} Given $H$ as in the previous item, classify all connected graded
  Hopf algebras $R$ in ${}_H^H\mathcal{YD}$.
  \item {\text{\rm Question\,\uppercase\expandafter{\romannumeral3} (\cite{AC13})}} Given $H$ and $R$ as in previous items, classify all liftings, that is, classify all Hopf algebras $A$ such that $\text{gr}\,A\cong R\sharp H$.
\end{itemize}

If $A$ is a Hopf algebra satisfying $A_{[0]}$=$H$, where $H$ is an arbitrary finite-dimensional Hopf algebra, then we also call $A$ is a Hopf algebra over $H$.

Following this generalized lifting method, G. A. Garcia and J. M. J. Giraldi \cite{GG16} determined all
finite-dimensional Hopf algebras over a Hopf algebra of dimension $8$ without the Chevalley property, and the corresponding infinitesimal braidings areirreducible objects and obtained some new Hopf algebras of dimension $64$.
 Motivated by their work, the authors in \cite{HX16} found all Hopf algebras of dimension $72$ over a Hopf algebra of dimension $12$ without the Chevalley property, and the corresponding infinitesimal braidings are simple objects.

The present paper is in some sense a sequel to \cite{AC13,GG16, HX16} and classifies all finite-dimensional Hopf algebras over a Hopf algebra $H$ of dimension $16$ without the Chevalley property. Here $H$ as an algebra is generated by the elements $a$, $b$, $c$, $d$ satisfying
\begin{gather*}
a^4=1,\quad b^2=0,\quad c^2=0, \quad d^4=1,\quad a^2d^2=1,\quad ad=da,\quad bc=0=cb,\\
  ab=\xi ba,\quad ac=\xi ca,\quad bd=\xi db,\quad cd=\xi dc,\quad bd=ca,\quad
  ba=cd.
\end{gather*}
and the coalgebra structure is given by
\begin{align*}
\Delta(a)=a\otimes a+b\otimes c,\quad \Delta(b)=a\otimes b+b\otimes d,\\
  \Delta(c)=c\otimes a+d\otimes c,\quad \Delta(d)=d\otimes d+c\otimes b.
\end{align*}
In order to find new Nichols algebras in $\HYD$, we need to compute simple objects in $\HYD$. Following the fact that the category ${}_{\D}\mathcal{M}$ is equivalent to the category $\HYD$, we first determine the structure of the Drinfeld double
$\D=\D(H^{cop})$ and study the irreducible representations of $\D$. In fact, we show in Theorem
$\ref{thmsimplemoduleD}$, there exist $16$ one-dimensional modules
$\K_{\chi_{i,j,k}}$ for $0\leq i,j<2,0\leq k<4$ and $36$ two-dimensional modules
$V_{i,j,k,\iota}$ for $(i,j,k,\iota)\in \Lambda$. Then we translate the simple $\D$-modules to the simple objects in $\HYD$ and describe explicitly their structures as Yetter-Drinfeld modules and their braidings. Using the braidings, we show that there exist $16$ Nichols algebas of dimension $8$ that are isomorphic to quantum linear spaces as algebras but as coalgebras are more complicated. In fact, we have
\begin{mainthm}
 The Nichols algebra $\BN(V)$ over a simple object $V$ in $\HYD$ is finite-dimensional if and only if $V$ is isomorphic either to $\K_{\chi_{i,j,k}}$ with $(i,j,k)\in\Lambda^0$ or $V_{i,j,k,\iota}$ with $(i,j,k,\iota)\in\cup_{3\leq\ell\leq 6}\Lambda^{\ell}$.
\end{mainthm}
 Moreover, we present them by generators and relations (see section $\ref{secNicholsalgebaH}$ for more details). To our knowledge, these Nichols algebras are new. Finally, we need to study the bosonizations of these Nichols algebras and their deformations. As a consequence, we obtain some nontrivial liftings in the following
\begin{mainthm}\label{thm:A}\label{liftingsoverH}
Let $A$ be a finite-dimensional Hopf algebra over $H$ such that the infinitesimal braiding is a simple object $V\in\HYD$. Then $A$ is isomorphic to one of the following
\begin{itemize}
  \item $\bigwedge\K_{\chi_{i,j,k}}\sharp H$ with $(i,j,k)\in\Lambda^0$, which has dimension $32$;
  \item $\BN(V_{i,j,k,\iota})\sharp H$ with $(i,j,k,\iota)\in \Lambda^3\cup\Lambda^4$, which has dimension $128$;
  \item $\LA$ with $(i,j,k,\iota)\in \Lambda^5$ and some $\mu\in\K$, which has dimension $128$;
  \item $\LAA$ with $(i,j,k,\iota)\in \Lambda^6$ and some $\mu\in\K$, which has dimension $128$.
\end{itemize}
\end{mainthm}

This paper is organized as follows:
In section $\ref{Preliminary}$ we first recall some basic definitions and facts about Yetter-Drinfeld module,
 Nichols algebra, Drinfeld double and Hopf $2$-cocycle. In section $\ref{secDrinfelddouble}$,
 we give a detailed description of the Hopf algebra structure of $H$ and determine the Drinfeld
 double $\D=\D(H^{cop})$ in terms of generators and relations. In section $\ref{secPresentation}$,
 we study the irreducible representations of the Drinfeld double $\D(H^{cop})$. In section $\ref{secHYD}$, we determine the simple objects in $\HYD$ by using the equivalence ${}_{\D}\mathcal{M}\simeq \HYD$, and also describe
 their braidings. In section $\ref{secNicholsalgebaH}$, we study the Nichols algebras over simple objects in $\HYD$. We determine all finite-dimensional Nichols algebras and describe their structures in terms of generators and relations.
 In section $\ref{secHopfalgebraH}$, we determine all finite-dimensional Hopf algebras $A$ over $H$ and the corresponding infinitesimal braidings are simple objects in $\HYD$.  In particular,  we figure out $8$ classes of new Hopf algebras of
 dimension $128$ without the Chevalley property.

\section{Preliminaries}\label{Preliminary}
\paragraph{Conventions.} Throughout the paper, the ground field $\K$ is an algebraically closed field of characteristic zero. Our references for Hopf algebra theory are \cite{M93} and \cite{R11}.

 The notation for a Hopf algebra $H$ is standard: $\Delta$, $\epsilon$, and $S$ denote the comultiplication, the counit and the antipode.
 We use Sweedler's notation for the comultiplication and coaction, for example, for any $h\in H$, $\Delta(h)=h_{(1)}\otimes h_{(2)}$,
 $\Delta^{(n)}=(\Delta\otimes id^{\otimes n})\Delta^{(n-1)}$. Given a Hopf algebra $H$ with bijective antipode, we denote by $H^{op}$
 the Hopf algebra with the opposite multiplication, $H^{cop}$ the Hopf algebra with the opposite comultiplication, and $H^{bop}$ the Hopf algebra
 $H^{op\,cop}$. If $V$ is a $\K$-vector space, $v\in V$ and $f\in V\As$, we use either $f(v)$, $\langle f$, $v\rangle$, or $\langle v,
 f\rangle$ to denote the evaluation.
For any $n>1$, $M_n(\K)$ and $M_n\As(\K)$ denote matrix algebra and matrix coalgebra. If a simple coalgebra $C\cong M_n\As(\K)$,
we call the basis $(c_{ij})_{1\leq i,j\leq n}$ a comatrix basis if $\Delta(c_{i,j})=\sum_{k=1}^{n}c_{ik}\otimes c_{kj}$ and $\epsilon(c_{ij})=\delta_{i,j}$.
\subsection{Yetter-Drinfeld module and Nichols algebra}
\begin{defi}
Let $H$ be a Hopf algebra with bijective antipode. A Yetter-Drinfeld module over $H$ is a left $H$-module and a left $H$-comodule with comodule structure denoted by $\delta: V\mapsto H\otimes V, v\mapsto v_{(-1)}\otimes v_{(0)}$, such that
\begin{align*}
\delta(h\cdot v)=h_{(1)}v_{(-1)}S(h_{(3)})\otimes h_{(2)}v_{(0)},
\end{align*}
for all $v\in V,h\in H$. Let ${}^{H}_{H}\mathcal{YD}$ be the
category of Yetter-Drinfeld modules over $H$ with $H$-linear and
$H$-colinear maps as morphisms.
\end{defi}

The category ${}^{H}_{H}\mathcal{YD}$ is monoidal, braided. Indeed, if $V,W\in {}^{H}_{H}\mathcal{YD}$, $V\otimes W$ is the
tensor product over $\mathbb{C}$ with the diagonal action and coaction of $H$ and braiding
\begin{align}\label{equbraidingYDcat}
c_{V,W}:V\otimes W\mapsto W\otimes V, v\otimes w\mapsto v_{(-1)}\cdot w\otimes v_{(0)},\forall\,v\in V, w\in W.
\end{align}
Moreover, ${}^{H}_{H}\mathcal{YD}$ is rigid. That is, it has left dual and right dual, if we take $V\As$ and ${}\As V$ to the dual of
$V$ as vector space, then the left dual and the right dual are defined by
\begin{align*}
\langle h\cdot f,v\rangle=\langle f,S(h)v\rangle,\quad f_{(-1)}\langle f_{(0)},v\rangle=S^{-1}(v_{(-1)})\langle f, v_{(0)}\rangle,\\
\langle h\cdot f,v\rangle=\langle f,S^{-1}(h)v\rangle,\quad f_{(-1)}\langle f_{(0)},v\rangle=S(v_{(-1)})\langle f, v_{(0)}\rangle.
\end{align*}
We consider Hopf algebra in ${}^{H}_{H}\mathcal{YD}$. If $R$ is a Hopf algebra in ${}^{H}_{H}\mathcal{YD}$,
the space of primitive elements $P(R)=\{x\in R|\delta(x)=x\otimes 1+1\otimes x\}$ is a Yetter-Drinfeld
submodule of $R$. Moreover, for any finite-dimensional graded Hopf algebra in ${}^{H}_{H}\mathcal{YD}$, it satisfies the Poincar\'{e} duality:
\begin{pro}$\cite[Proposition\,3.2.2]{AG99}$
Let $R=\oplus_{n=0}^N R(i)$ be a finite-dimensional graded Hopf algebra in ${}^{H}_{H}\mathcal{YD}$, and suppose that $R(N)\neq 0$. Then $\dim R(i)=\dim R(N-i)$ for any $0\leq i<N$.
\end{pro}
\begin{defi}
Let $V\in{}^{H}_{H}\mathcal{YD}$ and $I(V)\subset T(V)$ be the largest $\mathds{N}$-graded ideal and coideal
such that $I(V)\cap V=0$. We call $\mathcal{B}(V)=T(V)/I(V)$ the Nichols algebra of $V$. Then $\mathcal{B}(V)=\oplus_{n\geq 0}\mathcal(B)^n(V)$
is an $\mathds{N}$-graded Hopf algebra in ${}^{H}_{H}\mathcal{YD}$.
\end{defi}
\begin{lem}$\cite{AS02}$
The Nichols algebra of an object $V\in{}^{H}_{H}\mathcal{YD}$ is the
(up to isomorphism) unique $\mathds{N}$-graded Hopf algebra $R$ in
${}^{H}_{H}\mathcal{YD}$ satisfying the following properties:
\begin{align*}
&R(0)=\mathds{k}, \quad R(1)=V,\\
&R(1)\ \text{\ generates the algebra R},\\
&P(R)=V.
\end{align*}
\end{lem}

Nichols algebras play a key role in the classification of pointed Hopf algebras, and
we close this subsection by giving the explicit relation between $V$ and $V\As$ in ${}^{H}_{H}\mathcal{YD}$.
\begin{pro}$\cite[Proposition\,3.2.30]{AG99}$\label{proNicholsdual}
Let $V$ be an object in ${}_H^H\mathcal{YD}$. If $\BN(V)$ is finite-dimensional, then $\BN(V\As)\cong \BN(V)\As$.
\end{pro}

\subsection{Radford biproduct construction}
Let $R$ be a bialgebra (resp. Hopf algebra) in ${}^{H}_{H}\mathcal{YD}$ and denote the coproduct
by $\Delta_R(r)=r^{(1)}\otimes r^{(2)}$. We define the Radford biproduct $R\#H$. As a vector space,
$R\#H=R\otimes H$ and the multiplication and comultiplication are given by the smash product and smash-coproduct, respectively:
\begin{align}
(r\#g)(s\#h)&=r(g_{(1)}\cdot s)\#g_{(2)}h,\\
\Delta(r\#g)&=r^{(1)}\#(r^{(2)})_{(-1)}g_{(1)}\otimes (r^{(2)})_{(0)}\#g_{(2)}.
\end{align}
Clearly, the map $\iota:H\rightarrow R\#H, h\mapsto 1\# h,\ \forall h\in H$, and the map
$\pi:R\#H\rightarrow H,r\#h\mapsto \epsilon_R(r)h,\ \forall r\in R, h\in H$ such that $\pi\circ\iota=id_H$. Moreover, $R=(R\#H)^{coH}$.

Let $R, S$ be bialgebras (resp. Hopf algebra) in ${}^{H}_{H}\mathcal{YD}$ and $f:R\rightarrow S$ be a bialgebra
morphisms in ${}^{H}_{H}\mathcal{YD}$. $f\#id:R\#H\rightarrow S\#H$ defined by $(f\#id)(r\#h)=f(r)\#h,
\forall r\in R, h\in H$. In fact, $R\rightarrow R\#H$ and $f\mapsto f\# id$ describes a
functor from the category of bialgebras (resp. Hopf algebras) in ${}^{H}_{H}\mathcal{YD}$ and their morphisms
to the category of usual bialgebras (resp. Hopf algebras).

Conversely, if $A$ is a bialgebra (resp. Hopf algebra) and $\pi:A\rightarrow H$ a bialgebra admitting a bialgebra
section $\iota:H\rightarrow A$ such that $\pi\circ\iota=id_H$ (we call $(A,H)$ a Radford pair for convenience),
$R=A^{coH}=\{a\in A\mid(id\otimes\pi)\Delta(a)=a\otimes 1\}$ is a bialgebra (resp. Hopf algebra) in ${}^{H}_{H}\mathcal{YD}$
and $A\simeq R\#H$, whose Yetter-Drinfeld module and coalgebra structures are given by:
\begin{align*}
h\cdot r&=h_{(1)}rS_A(h_{(2)}),\quad
\delta(r)=(\pi\otimes id)\Delta_A(r),\\
\Delta_R(r)&=r_{(1)}(\iota S_H(\pi(r_{(2)})))\otimes r_{(3)},\quad
\epsilon_R=\epsilon_A|_R,\\
S_R(r)&=(\iota\pi(r_{(1)}))S_A(r_{(2)}),\quad\text{if $A$ is a Hopf algebra}.
\end{align*}

\subsection{Drinfeld double}
\begin{defi}
Let $H$ be a finite-dimensional Hopf algebra with bijective antipode $S$ over $\K$. The Drinfeld double $\D(H)=H^{\ast\,cop}\otimes H$
is a Hopf algebra with the tensor product coalgebra structure and algebra structure defined by
\begin{align}\label{equDrinfelddouble}
(p\otimes a)(q\otimes b)=p\langle q_{(3)}, a_{(1)}\rangle q_{(2)}\otimes a_{(2)}\langle q_{(1)}, S^{-1}(a_{(3)})\rangle.
\end{align}
\end{defi}

By $\cite[Proposition 10.6.16]{M93}$, the category
${}_{\D(H)}{\M}$ of left modules is equivalent to the category ${}_H\mathcal{YD}^H$ of Yetter-Drinfeld
modules. But ${}_H\mathcal{YD}^H$ is equivalent
to the category ${}_{H^{cop}}^{H^{cop}}\mathcal{YD}$ of Yetter-Drinfeld modules.
Thus we have the following
result.
\begin{pro}$\cite{M93}$\label{proDouble}
Let $H$ be a finite-dimensional Hopf algebra with bijective antipode $S$ over $\K$.
Then the category ${}_{\D(H^{cop})}\M$ of left modules is equivalent to the category ${}_H^H\mathcal{YD}$
of Yetter-Drinfeld modules.
\end{pro}

\subsection{Hopf $2$-cocycle deformation}
Let $(H,m,1,\Delta,\epsilon,S)$ be a Hopf algebra. The convolution invertible bilinear form $\sigma: H\otimes H\mapsto
\K$ is called a (left) normalized Hopf $2$-cocycle of $H$ if
\begin{gather*}
\sigma(a,1)=\sigma(1,a)=\epsilon(a),\quad \forall a\in H,\\
\sum\sigma(a_{(1)}, b_{(1)})\sigma(a_{(2)}b_{(2)},c)=\sum\sigma(b_{(1)}, c_{(1)})\sigma(a, b_{(2)}c_{(2)}),\quad\forall a,b,c\in H.
\end{gather*}
Denote by $\sigma^{-1}$ the convolution inverse of $\sigma$. We can construct a new Hopf algebra $(H^{\sigma},m^{\sigma},1,\Delta,\epsilon, S^{\sigma})$, where $H^{\sigma}=H$ as coalgebras,and
\begin{align*}
m^{\sigma}(a\otimes b)=\sum\sigma(a_{(1)},b_{(1)})a_{(2)}b_{(2)}\sigma^{-1}(a_{(3)},b_{(3)}),\quad\forall a,b\in H,\\
S^{\sigma}(a)=\sum\sigma(a_{(1)},S(a_{(2)}))S(a_{(3)})\sigma^{-1}(S(a_{(4)}),a_{(5)}),\quad\forall a\in H.
\end{align*}
We denote by $\mathcal{Z}^2(H,\mathds{k})$ the set of normalized Hopf $2$-cocycles on $H$.
And we will use the following equivalence of categories of Yetter-Drinfeld modules which is due originally to Majid and Oeckl \cite[Theorem\;2.7]{MO99}.
\begin{defi}
Let $M$ be a Yetter-Drinfeld module over $H$ and $\sigma$ is a Hopf $2$-cocycle of $H$. Then there is a corresponding Yetter-Drinfeld module over $H^{\sigma}$ denoted by $M^{\sigma}$ defined as: it is $M$ as a comodule, and the $H^{\sigma}$-action is given by
\begin{align*}
a\cdot^{\sigma}m=\sum \sigma((a_{(2)}\cdot m_{(0)})_{(1)},a_{(1)})(a_{(2)}\cdot m_{(0)})_{(0)}\sigma^{-1}(a_{(3)}\otimes m_{(1)}),
\end{align*}
\end{defi}\noindent
for all $a\in H$ and $m\in M$.
\begin{pro}\cite[Theorem\;2.7]{MO99}\label{proCocycledeformation}
Let $\sigma$ be a Hopf $2$-cocycle on the Hopf algebra $H$. Then the categories $\HYD$ and ${}_{H^{\sigma}}^{H^{\sigma}}\mathcal{YD}$ are monoidally equivalent under the functor
\begin{align*}
F_{\sigma}:\HYD\mapsto {}_{H^{\sigma}}^{H^{\sigma}}\mathcal{YD},
\end{align*}
which is the identity on homomorphisms and on the objects is given by $F_{\sigma}(M)=M^{\sigma}$.
\end{pro}

\section{A Hopf algebra H of dimension $16$ and the Drinfeld double $\D$}\label{secDrinfelddouble}
In this section we describe explicitly the structure of $H$ and present the Drinfeld double $\D=\D(H^{cop})$ by generators and relations.

Throughout the paper, we fix $\xi$ a primitive $4$-th root of unity. For the classification of Hopf algebra of dimension $16$, the semisimple case was classified by Y.~Kashina \cite{K00}, the pointed nonsemisimple case was given by S.~Caenepeel, S.~D\u{a}sc\u{a}lescu, and S.~Raianu \cite{CDR00}, and the full classification was done by G. A. Garc\'{\i}a and C. Vay \cite{GV10}. Now, we choose one pointed Hopf algebra listed in \cite[Section\;2.5]{CDR00}
\begin{defi}
\begin{gather*}
\A:=\langle g,h,x\mid g^4=1, h^2=1, hg=gh, hx=-xh,gx=xg, x^2= 1-g^2 \rangle.\\
\Delta(g)=g\otimes g,\quad \Delta(h)=h\otimes h, \quad\Delta(x)=x\otimes 1+gh\otimes x.
\end{gather*}
\end{defi}
\begin{rmk}\label{rmkGAcocycledefor}
Let $\Gamma\cong Z_4\times Z_2$ be an abelian group with generators $g,h$ and $V=\K\{v\}\in {}_{\Gamma}^{\Gamma}\mathcal{YD}$ given by
\begin{align*}
h\cdot v=\Lam(h)v=-v,\quad g\cdot v=\Lam(g)v=v;\quad\delta(v)=gh\otimes v.
\end{align*}
Then the Nichols algebra $\BN(V)=\K(v)/(v^2)$. Let $G=\BN(V)\sharp \K\Gamma$, then $G$ is a Hopf algebra generated as an algebra by $g$, $h$ and $x$ satisfying
\begin{align*}
g^4=1,\quad h^2=1,\quad hg=gh,\quad hx=-xh,\quad gx=xg,\quad x^2=0,
\end{align*}
and the coalgebra structure given by
\begin{align*}
\Delta(g)=g\otimes g,\quad \Delta(h)=h\otimes h,\quad \Delta(x)=x\otimes 1+gh\otimes x.
\end{align*}
In particular, $A$ is a lifting of the Nichols algebra $\BN(V)$ by deforming the relation $x^2=0$. By \cite[Theorem\;A.1]{Ma08}, $A$ is a Hopf $2$-cocycle deformation of $G$. That is, there exists some $\sigma\in \mathcal{Z}^2(G,\mathds{k})$ such that $A\cong G^{\sigma}$.
\end{rmk}

In order to describe explicitly the structure of $H$ as the dual Hopf algebra of $A$, we first need to compute the irreducible representations of $A$.
\begin{lem}\label{lem1}
There are four one-dimensional $\A-$modules denoted by $\chi_{i,j}$, $i,j\in Z_2$ given by
\begin{align*}
\chi_{i,j}(g)=(-1)^i,\quad \chi_{i,j}(h)=(-1)^j,\quad \chi_{i,j}(x)=0.
\end{align*}
and two two-dimensional simple $\A-$modules denoted by $\rho_i$, $i\in Z_2$ given by
\begin{align*}
   \rho_1(g)&=\left(\begin{array}{ccc}
                                   \xi & 0\\
                                   0 & \xi
                                 \end{array}\right),\quad
    \rho_1(h)=\left(\begin{array}{ccc}
                                   1 & 0\\
                                   0 & -1
                                 \end{array}\right),\quad
    \rho_1(x)=\left(\begin{array}{ccc}
                                   0 & \sqrt 2\\
                             \sqrt 2& 0
                                 \end{array}\right);\\
    \rho_2(g)&=\left(\begin{array}{ccc}
                                   \xi^3 & 0\\
                                   0 & \xi^3
                                 \end{array}\right),\quad
    \rho_2(h)=\left(\begin{array}{ccc}
                                   1 & 0\\
                                   0 & -1
                                 \end{array}\right),\quad
    \rho_2(x)=\left(\begin{array}{ccc}
                                   0 & \sqrt 2\\
                             \sqrt 2& 0
                                 \end{array}\right).
\end{align*}
\end{lem}

Let $(\K^2, \rho_i)_{i=1,2}$ be the $2$-dimensional representations given in Lemma $\ref{lem1}$. Let ${(E_{ij})}_{i,j=1,2}$ be the coordinate functions of $\M(2,\K)$. And let $c_{ij}:=E_{ij}\circ \rho_1, d_{ij}:=E_{ij}\circ \rho_2$, we can regard $\E_C:=\{C_{ij}\}_{i,j=1,2}$ and $\E_D:=\{D_{ij}\}_{i,j=1,2}$ as comatrix basis of the simple subcoalgebras of $H$ isomorphic to $C$ and $D$ respectively. The following Lemma shows some useful relations of the elements of $\E_C$ and $\E_D$ and the proof is much similar with that in \cite[Lemma\;3.3]{GV10}.
\begin{lem}
The elements of $\E_C$ and $\E_D$ satisfy:
\begin{align*}
S(C_{12})=\xi D_{21},\quad S(C_{21})=\xi^3 D_{12}, \quad
S(C_{11})=D_{11},\quad S(C_{22})=D_{22},\\
S(D_{12})=\xi^3 C_{21},\quad S(D_{21})=\xi C_{12},\quad
S(D_{11})=C_{11},\quad S(D_{22})=C_{22},\\
C_{11}^2=C_{22}^2=\chi_{1,0},\quad C_{11}C_{22}=C_{22}C_{11}=\chi_{1,1},\,
C_{11}C_{12}=\xi C_{12}C_{11},\\
C_{12}^2=0=C_{21}^2,\quad C_{12}C_{21}=0=C_{21}C_{12}, \quad C_{11}C_{21}=\xi C_{21}C_{11}, \\
C_{22}C_{12}=-\xi C_{12}C_{22},\quad C_{22}C_{21}=-\xi C_{21}C_{22},\quad C_{12}C_{22}=C_{21}C_{11},\\ C_{12}C_{11}=C_{21}C_{22},\quad C_{11}C_{12}=-C_{22}C_{21},\quad C_{22}C_{12}=-C_{11}C_{21}.
\end{align*}
\end{lem}

\begin{rmk}\label{rmkHindependent}
After an easy computation, the elements
\begin{gather*}
C_{11}^3,\, C_{22}^3,\,
C_{11}^2=\chi_{1,0},\, C_{11}C_{22}=\chi_{1,1},\, C_{22}^2C_{11}^2=\epsilon,\,
C_{11}^3C_{22}=\chi_{0,1},\\
C_{11},\, C_{22},\,
C_{12},\, C_{21},\, C_{11}C_{12},\, C_{11}C_{21},\, C_{11}^2C_{12},\, C_{11}^2C_{21},\,
C_{11}^3C_{12},\, C_{11}^3C_{21}.
\end{gather*}
are linearly independent. Thus $H$ is generated by the simple subcoalgebra $C$ since $\dim H=\dim \A=16$.
\end{rmk}

Now for convenience, let $a=C_{11}$, $b=C_{12}$, $c=C_{21}$ and $d=C_{22}$. Then we have the following Proposition.
\begin{pro}\label{proStructureofH}
\begin{enumerate}
  \item $H$ as an algebra is generated by the elements $a$, $b$, $c$, $d$ satisfying the relations
  \begin{gather*}
  a^4=1,\quad b^2=0,\quad c^2=0, \quad d^4=1,\quad a^2d^2=1,\quad ad=da,\quad bc=0=cb,\\
  ab=\xi ba,\quad ac=\xi ca,\quad bd=\xi db,\quad cd=\xi dc,\quad bd=ca,\quad
  ba=cd.
  \end{gather*}
  \item A linear basis of $H$ is given by
  \begin{align*}
  \{1, a,\,a^2,\, a^3,\,d, \, da,\, da^2,\, da^3,\, b,\, c,\, ba,\, ca,\, ba^2,\, ca^2,\, ba^3,\, ca^3\}.
  \end{align*}
  \item The coalgebra structure of $H$ is given by
  \begin{align*}
  \Delta(a)=a\otimes a+b\otimes c,\quad \Delta(b)=a\otimes b+b\otimes d,\\
  \Delta(c)=c\otimes a+d\otimes c,\quad \Delta(d)=d\otimes d+c\otimes b,\\
  \Delta(a^2)=a^2\otimes a^2,\quad \Delta(a^3)=a^3\otimes a^3+ba^2\otimes ca^2,\\
  \Delta(da)=da\otimes da,\quad \Delta(da^2)=da^2\otimes da^2+ca^2\otimes ba^2,\\
  \Delta(da^3)=da^3\otimes da^3,\quad \Delta(ba)=ba\otimes da+a^2\otimes ba,\\
  \Delta(ca)=ca\otimes a^2+da\otimes ca,\, \Delta(ba^2)=ba^2\otimes da^2+a^3\otimes ba^2,\\
  \Delta(ca^2)=ca^2\otimes a^3+da^2\otimes ca^2,\,\Delta(ba^3)=ba^3\otimes da^3+1\otimes ba^3,\\
  \Delta(ca^3)=ca^3\otimes 1+da^3\otimes ca^3,\quad \Delta(1)=1\otimes 1, \\
  \epsilon(a)=1=\epsilon(d),\quad \epsilon(b)=0=\epsilon(c).
  \end{align*}
  \item The antipode of $H$ is given by
  \begin{align*}
  S(a)=a^3,\quad S(d)=d^3,\quad S(b)=\xi ca^2, \quad S(c)=\xi^3ba^2.
  \end{align*}
\end{enumerate}
\end{pro}

\begin{rmk}\label{rmkHdualtoA}
Denote by $\{(a^i)\As,(ba^i)\As,(ca^i)\As,(da^i)\As,\;0\leq i\leq 3\}$ the basis of the dual Hopf algebra $H\As$.
Let
\begin{gather*}
\widetilde{x}=\sum_{i=0}^3 (ba^i)^{\ast}+\sum_{i=0}^3 (ca^i)^{\ast},\\
\widetilde{g}=\sum_{i=0}^3 \xi^{i}(a^i)^{\ast}+\sum_{i=0}^3 \xi^{i+1}(da^i)^{\ast},\quad
\widetilde{h}=\sum_{i=0}^3 (a^i)^{\ast}-\sum_{i=0}^3(da^i)^{\ast}.
\end{gather*}
Then using the multiplication table induced by the relations of $H$ given in Proposition $\ref{proStructureofH}$ and after a tedious computation, we have that
\begin{gather*}
\widetilde{g}^4=1,\quad \widetilde{h}^2=1,\quad \widetilde{h}\widetilde{g}=\widetilde{g}\widetilde{h},\quad \widetilde{g}\widetilde{x}=\widetilde{x}\widetilde{g},\quad
\widetilde{h}\widetilde{x}=-\widetilde{x}\widetilde{h},\\
\Delta(\widetilde{x})=\widetilde{x}\otimes \epsilon+\widetilde{g}\widetilde{h}\otimes \widetilde{x},\quad
\Delta(\widetilde{g})=\widetilde{g}\otimes \widetilde{g},\quad
\Delta(\widetilde{h})=\widetilde{h}\otimes \widetilde{h}.
\end{gather*}
In particular, $G(H\As)\cong Z_4\times Z_2$ with generators $\widetilde{g}$ and $\widetilde{h}$.
\end{rmk}

In order to compute the structure of the Drinfeld double $D(H^{cop})$ of $H^{cop}$ in terms of generators and relations, we have the following Lemmawhich builds the isomorphism $\A\cong H^{\ast}$ explicitly.
\begin{lem}\label{lemAtoHdual}
The algebra map $\psi:\A\mapsto H^{\ast}$ given by
\begin{align*}
\psi(g)&=\widetilde{g}=\sum_{i=0}^3 \xi^{i}(a^i)^{\ast}+\sum_{i=0}^3 \xi^{i+1}(da^i)^{\ast},\\
\psi(h)&=\widetilde{h}=\sum_{i=0}^3 (a^i)^{\ast}-\sum_{i=0}^3(da^i)^{\ast},\\
\psi(x)&=\sqrt 2\widetilde{x}=\sqrt 2\sum_{i=0}^3 (ba^i)^{\ast}+\sqrt 2\sum_{i=0}^3 (ca^i)^{\ast}.
\end{align*}
is a Hopf algeba isomorphism.
\end{lem}
\begin{proof}
By Remark $\ref{rmkHdualtoA}$, $\psi$ is a coalgebra map and $\psi(A)$ contains properly $G(H\As)$. Thus by the Nichols-Zoeller theorem, $\psi$ is epimorphic. Since $\dim A=\dim H\As=16$, $\psi$ is isomorphic.
\end{proof}

\begin{rmk}\label{rmkAtoHdual}
Let $\{ g^j, g^jh, g^jx, g^jhx\}_{0\leq j<4}$ be a linear basis of $\A$. We have
\begin{align*}
\psi(g^j)=\sum_{i=0}^3 \xi^{ij}(a^i)^{\ast}+\sum_{i=0}^3 \xi^{ij+j}(da^i)^{\ast},\\
\psi(g^jh)=\sum_{i=0}^3 \xi^{ij}(a^i)^{\ast}-\sum_{i=0}^3 \xi^{ij+j}(da^i)^{\ast},\\
\psi(g^jx)=\sqrt 2\sum_{i=0}^3 \xi^{ij+j}(ba^i)^{\ast}+\sqrt 2\sum_{i=0}^3 \xi^{ij+j}(ca^i)^{\ast},\\
\psi(g^jhx)=\sqrt 2\sum_{i=0}^3 \xi^{ij+j}(ba^i)^{\ast}-\sqrt 2\sum_{i=0}^3\xi^{ij+j} (ca^i)^{\ast}.
\end{align*}
\end{rmk}

Now we try to describe the Drinfeld double $\D:=\D(H^{cop})$ of $H^{cop}$.
\begin{pro}
$\D$ as a $\K$-coalgebra is isomorphic to $\A^{bop}\otimes H^{cop}$, and as a $\K$-algebra is generated by the elements $g$, $h$, $x$, $a$, $b$, $c$, $d$ satisfying the relations in $H^{cop}$, the relations in $\A^{bop}$ and
\begin{align*}
ag=ga,\quad ah=ha,\quad dg=gd,\quad dh=hd,\\
bg=gb,\quad bh=-hb,\quad cg=gc,\quad ch=-hc,\\
ax+\xi xa=\sqrt 2\xi(c-ghb),\, dx-\xi xd=\sqrt 2\xi(ghc-b),\\
bx+\xi xb=\sqrt 2\xi(d-gha),\, cx-\xi xc=\sqrt 2\xi(ghd-a).
\end{align*}
\end{pro}
\begin{proof}
Note that
\begin{align*}
\Delta_{\A^{bop}}^{2}(g)&=g\otimes g\otimes g,\quad
\Delta_{\A^{bop}}^{2}(h)=h\otimes h\otimes h,\\
\Delta_{\A^{bop}}^{2}(x)&=1\otimes 1\otimes x+1\otimes x\otimes gh+x\otimes gh \otimes gh,\\
\Delta_{H^{cop}}^{2}(a)&=a\otimes a\otimes a+a\otimes c\otimes b+c\otimes b\otimes a+c\otimes d\otimes b,\\
\Delta_{H^{cop}}^{2}(b)&=b\otimes a\otimes a+b\otimes c\otimes b+d\otimes b\otimes a+d\otimes d\otimes b,\\
\Delta_{H^{cop}}^{2}(c)&=a\otimes a\otimes c+a\otimes c\otimes d+c\otimes d\otimes d+c\otimes b\otimes c,\\
\Delta_{H^{cop}}^{2}(d)&=d\otimes d\otimes d+d\otimes b\otimes c+b\otimes a\otimes c+b\otimes c\otimes d.
\end{align*}
We have that
\begin{align*}
ag&=\langle g,a\rangle ga\langle g,S(a)\rangle=ga,\quad
ah=\langle h,a\rangle ha\langle h,S(a)\rangle=ha,\\
dg&=\langle g,d\rangle gd\langle g,S(d)\rangle=gd,\quad
dh=\langle h,d\rangle hd\langle h,S(d)\rangle=hd,\\
bg&=\langle g,d\rangle gb\langle g,S(a)\rangle=gb,\quad
bh=\langle h,d\rangle hb\langle h,S(a)\rangle=-hb,\\
cg&=\langle g,a\rangle gc\langle g,S(d)\rangle=gc,\quad
ch=\langle h,a\rangle hc\langle h,S(d)\rangle=-hc,\\
ax&=\langle 1,a\rangle c\langle x,S(b)\rangle+\langle 1,a\rangle xa\langle gh,S(a)\rangle
+\langle x,c\rangle ghb\langle gh,S(a)\rangle   \\&
=\sqrt 2\xi c-\xi xa-\sqrt 2\xi ghb,\\
dx&=\langle 1,d\rangle b\langle x,S(c)\rangle+\langle 1,d\rangle xd\langle gh,S(d)\rangle
+\langle x,b\rangle ghc\langle gh,S(d)\rangle   \\&
=\sqrt 2\xi^3 b+\xi xd+\sqrt 2\xi ghc,\\
bx&=\langle 1,d\rangle d\langle x,S(b)\rangle+\langle 1,d\rangle xb\langle gh,S(a)\rangle
+\langle x,b\rangle gha\langle gh,S(a)\rangle   \\&
=\sqrt 2\xi d-\xi xb-\sqrt 2\xi gha,\\
cx&=\langle 1,a\rangle a\langle x,S(c)\rangle+\langle 1,a\rangle xc\langle gh,S(d)\rangle
+\langle x,c\rangle ghd\langle gh,S(d)\rangle   \\&
=\sqrt 2\xi^3 a+\xi xc+\sqrt 2\xi ghd.
\end{align*}
\end{proof}

\section{Presentation of the Drinfeld double $\D$}\label{secPresentation}
In this section, we compute the irreducible representations of $\D$. We begin this section by describing the one-dimensional $\D$-modules.
\begin{lem}\label{lemOnesimpleD}
There are $16$ non-isomorphic one-dimensional simple modules $\K_{\chi_{i,j,k}}$ given by the characters $\chi_{i,j,k},\,0\leq i,j<2,0\leq k<4$, where
\begin{align*}
\chi_{i,j,k}(g)&=(-1)^i,\quad\chi_{i,j,k}(h)=(-1)^j,\quad \chi_{i,j,k}(x)=0,\\
\chi_{i,j,k}(a)&=\xi^k,\, \chi_{i,j,k}(b)=0,\,\chi_{i,j,k}(c)=0,\,\chi_{i,j,k}(d)=(-1)^i(-1)^j\xi^k.
\end{align*}
Moreover, any one-dimensional $\D$-module is isomorphic to $\K_{\chi_{i,j,k}}$ for some $0\leq i,j<2,0\leq k<4$.
\end{lem}
\begin{proof}
Let $\chi\in G(\D^{\ast})=\hom(\D,\K)$. Since $a^4=1=g^4$ and $d^4=1=h^2$, we have that $\chi(a)^4=1=\chi(g)^4$ and $\chi(d)^4=1=\chi(h)^2$. From $b^2=0$, $c^2=0$ and $hx=-xh$, we have that $\chi(b)=\chi(x)=\chi(c)=0$, and whence $\chi(g)^2=1$ since $x^2=1-g^2$. From the relation $bx+\xi xb=\sqrt 2\xi(d-gha)$, we have $\chi(d)=\chi(g)\chi(g)\chi(a)$. Thus $\chi$ is completely determined by $\chi(a)$, $\chi(g)$ and $\chi(h)$. Let $\chi(a)=\xi^k$ for some $k\in Z_4$, $\chi(g)=(-1)^i$, $\chi(h)=(-1)^j$. It is clear that these modules are pairwise non-isomorphic and any one-dimensional $\D$-module is isomorphic to $\K_{\chi_{i,j,k}}$ where $0\leq i,j<2,0\leq k<4$.
\end{proof}

Next, we describe two-dimensional simple $\D$-modules. For this, consider the finite set given by
\begin{align*}
\Lambda=\{(i,j,k,\iota)\mid i\in Z_4, j=1,3, k,\iota\in Z_2\}.
\end{align*}
Clearly, $|\Lambda|=32$.
\begin{lem}\label{lemTwosimpleD}
For any pair $(i,j,k)\in\Lambda$, there exists a simple left $\D$-module $V_{i,j,k}$ of dimension $2$. If we denote $\Lam_1=\xi^i$, $\Lam_2=\xi^j$, $\Lam_3=(-1)^k$ and $\Lam_4=(-1)^\iota$ the action on a fixed basis is given by
\begin{align*}
    [a]&=\left(\begin{array}{ccc}
                                   -\Lam_4\Lam_1 & 0\\
                                    0  & \xi \Lam_4\Lam_1
                                 \end{array}\right),\,
    [d]=\left(\begin{array}{ccc}
                                   \Lam_1 & 0\\
                                    0  & \xi \Lam_1
                                 \end{array}\right),\,
    [b]=\left(\begin{array}{ccc}
                                   0 & \Lam_4\\
                                   0 & 0
                                 \end{array}\right),\\
    [c]&=\left(\begin{array}{ccc}
                                   0 & 1\\
                                   0 & 0
                                 \end{array}\right),\,
    [g]=\left(\begin{array}{ccc}
                                   \Lam_2 & 0\\
                                   0   & \Lam_2
                                 \end{array}\right),\,
    [h]=\left(\begin{array}{ccc}
                                   \Lam_3 & 0\\
                                   0   & -\Lam_3
                                 \end{array}\right),\\
    [x]&=\left(\begin{array}{ccc}
                                   0 & \frac{\sqrt 2}{2}\xi \Lam_1^3(\Lam_2 \Lam_3 -\Lam_4)\\
                                   \sqrt 2\xi \Lam_1(\Lam_2\Lam_3+\Lam_4) & 0
                                 \end{array}\right),
\end{align*}
\end{lem}
\begin{proof}
Since the elements $g$, $h$, $a$, $d$ commute each other and $g^4=d^2=a^4=d^4=1$, we can choose a basis of the two dimensional simple $\D$-module $V$ such that the matrices defining the action on $V$ are of the form
\begin{align*}
    [g]&=\left(\begin{array}{ccc}
                                   g_1 & 0\\
                                   0   & g_2
                                 \end{array}\right),\,
    [h]=\left(\begin{array}{ccc}
                                   h_1 & 0\\
                                   0   & h_2
                                 \end{array}\right),\,
    [x]=\left(\begin{array}{ccc}
                                   x_1 & x_2\\
                                   x_3 & x_4
                                 \end{array}\right),\,
    [a]=\left(\begin{array}{ccc}
                                   a_1 & 0\\
                                    0  & a_2
                                 \end{array}\right),\\
    [d]&=\left(\begin{array}{ccc}
                                   d_1 & 0\\
                                    0  & d_2
                                 \end{array}\right),\quad
    [b]=\left(\begin{array}{ccc}
                                   b_1 & b_2\\
                                   b_3 & b_4
                                 \end{array}\right),\quad
    [c]=\left(\begin{array}{ccc}
                                   c_1 & c_2\\
                                   c_3 & c_4
                                 \end{array}\right),
\end{align*}
where $a_1^4=1=a_2^4$, $d_1^4=1=d_2^4$, $g_1^4=1=g_2^4$, $h_1^2=1=h_2^2$. From the relations $xh=-hx$, $bh=-hb$ and $ch=-hc$, we have that
\begin{align*}
x_1&=0=x_4, \quad(h_1+h_2)x_2=0=(h_1+h_2)x_3, \\
b_1&=0=b_4, \quad(h_1+h_2)b_2=0=(h_1+h_2)b_3, \\
c_1&=0=c_4, \quad(h_1+h_2)c_2=0=(h_1+h_2)c_3.
\end{align*}
If $h_1+h_2\neq 0$, then we have that $x_2=0=x_3$, $b_2=0=b_3$, $c_2=0=c_3$ and therefore $[b]$, $[c]$, $[x]$ are zero metrices, which implies that $V$ is can be decomposed as a $\D$-module, a contradiction. Thus we have $h_1=-h_2$.
From the relations $gx=xg$, $bg=gb$ and $cg=gc$, we have the relations
\begin{align*}
 (g_1-g_2)x_2=0=(g_1-g_2)x_3,\\
 (g_1-g_2)b_2=0=(g_1-g_2)b_3, \\
 (g_1-g_2)c_2=0=(g_1-g_2)c_3,
\end{align*}
which implies that $g_1=g_2$, otherwise $[b]$, $[c]$, $[x]$ are zero matrices and whence $V$ is not simple.

From the relations $b^2=0=c^2$ and $bc=0=cb$, we have that
\begin{align*}
b_2b_3=0=c_2c_3,\quad b_2c_3=0=b_3c_2,\quad c_2b_3=0=c_3b_2.
\end{align*}
And from the relations $ax+\xi xa=\sqrt 2\xi(c-ghb)$ and $dx-\xi xd=\sqrt 2\xi(ghc-b)$, we have the relations
\begin{align}
a_1x_2+\xi a_2x_2=\sqrt 2\xi(c_2-g_1h_1b_2),\quad a_2x_3+\xi a_1x_3=\sqrt 2\xi(c_3-g_2h_2b_3),\label{eq1}\\
d_1x_2-\xi d_2x_2=\sqrt 2\xi(g_1h_1c_2-b_2),\quad d_2x_3-\xi d_1x_3=\sqrt 2\xi(g_2h_2c_3-b_3).\label{eq2}
\end{align}
By permuting the elements of the basis, we may assume that $b_3=0=c_3$. Now we claim that $b_2\neq 0$ and $b_3\neq 0$. Indeed, if $b_2=0=b_3$, then it is clear that $V$ is simple if and only if $x_2x_3\neq 0$. By equations $\eqref{eq1}$, $\eqref{eq2}$, we have that
\begin{align*}
a_1x_2+\xi a_2x_2=0,\, a_2x_3+\xi a_1x_3=0,\, d_1x_2-\xi d_2x_2=0,\, d_2x_3-\xi d_1x_3=0,
\end{align*}
which imply that $a_1+\xi a_2=0$ and $a_2+\xi a_1=0$ and therefore $a_1=0=a_2$, a contradiction. Thus the claim follows. We may also assume that $c_2=1$.

From the relations $ab=\xi ba$, $ac=\xi ca$, $bd=\xi db$ and $cd=\xi dc$, we have that
\begin{align*}
(a_1-\xi a_2)b_2=0=(a_2-\xi a_1)b_3,\quad (d_2-\xi d_1)b_2=0=(d_1-\xi d_2)b_3,
\end{align*}
which implies that $a_1=\xi a_2$, $d_2=\xi d_1$.

From the relations $bd=ca$ and $ba=cd$, we have that
\begin{align*}
b_2d_2=a_2c_2,\quad b_3d_1=c_3a_1,\quad b_2a_2=c_2d_2,\quad b_3a_1=c_3d_1,
\end{align*}
which implies that $b_2^2=1$ and $a_2=b_2d_2$.

From the relations $bx+\xi xb=\sqrt 2\xi(d-gha)$ and $cx+\xi xc=\sqrt 2\xi(ghd-a)$, we have that
\begin{align*}
b_2x_3+\xi b_3x_2&=\sqrt 2\xi(d_1-g_1h_1a_1),\, b_3x_2+\xi b_2x_3=\sqrt 2\xi(d_2-g_2h_2a_2),\\
c_2x_3-\xi c_3x_2&=\sqrt 2\xi(g_1h_1d_1-a_1),\, c_3x_2-\xi c_2x_3=\sqrt 2\xi(g_2h_2d_2-a_2),
\end{align*}
which implies that $x_3=\sqrt 2\xi d_1(b_2+g_1h_1)$. By equations $\eqref{eq1}$ and $\eqref{eq2}$, we also have that $x_2=\frac{\sqrt 2}{2}\xi d_1^3(g_1h_1-b_2)$. Notice that from the relations $x^2=1-g^2$ and $a^2d^2=1$, we have that $x_2x_3=1-g_1^2$ and $a_1^2d_1^2=1=a_2^2d_2^2$. Indeed, since $a_2=b_2d_2$, $a_1=\xi a_2$ and $d_2=\xi d_1$, $a_1=-b_2d_1$ and therefore the relation $a_1^2d_1^2=1=a_2^2d_2^2$ holds. After a direct computation, we also have that the relation $x_2x_3=1-g_1^2$ holds.

From the discussion above, the matrices defining the action on V are of the form
\begin{align*}
    [a]&=\left(\begin{array}{ccc}
                                   -\Lam_4\Lam_1 & 0\\
                                    0  &  \xi \Lam_4\Lam_1
                                 \end{array}\right),\,
    [d]=\left(\begin{array}{ccc}
                                   \Lam_1 & 0\\
                                    0  & \xi \Lam_1
                                 \end{array}\right),\,
    [b]=\left(\begin{array}{ccc}
                                   0 & \Lam_4\\
                                   0 & 0
                                 \end{array}\right),\\
    [c]&=\left(\begin{array}{ccc}
                                   0 & 1\\
                                   0 & 0
                                 \end{array}\right),\,
    [g]=\left(\begin{array}{ccc}
                                   \Lam_2 & 0\\
                                   0   & \Lam_2
                                 \end{array}\right),\,
    [h]=\left(\begin{array}{ccc}
                                   \Lam_3 & 0\\
                                   0   & -\Lam_3
                                 \end{array}\right),\\
    [x]&=\left(\begin{array}{ccc}
                                   0 & \frac{\sqrt 2}{2}\xi \Lam_1^3(\Lam_2 \Lam_3 -\Lam_4)\\
                                   \sqrt 2\xi \Lam_1(\Lam_2\Lam_3+\Lam_4) & 0
                                 \end{array}\right),
\end{align*}
where, $\Lam_1^4=1$, $\Lam_2^4=1$, $\Lam_3^2=1$ and $\Lam_4^2=1$. And it is clear that $V$ is simple if and only if $\Lam_2^2\neq 1$. If we set $\Lam_1=\xi^i$, $\Lam_2=\xi^j$, $\Lam_3=(-1)^k$ and $\Lam_4=(-1)^\iota$, then we have that $j=1,3$ and $(i,j,k,\iota)\in\Lambda$.

Now we claim that $V_{i,j,k,\iota}\cong V_{p,q,r,\kappa}$ if and only if $(i,j,k,\iota)=(p,q,r,\kappa)$ in $\Lambda$.
Suppose that $\Phi:V_{i,j,k,\iota}\mapsto V_{p,q,r,\kappa}$ is an $\D$-module isomorphism, and denote by
$[\Phi]=(p_{i,j})_{i,j=1,2}$ the matrix of $\Phi$ in the given basis. As a module morphism, we have $[c][\Psi]=[\Psi][c]$ and $[a][\Psi]=[\Psi][a]$, which imply $p_{21}=0,\,p_{11}=p_{22}$ and $(\xi^p-\xi^i)p_{11}=0,\,(\xi^p-\xi^{i+1})p_{12}=0$. Thus we have $\xi^i=\xi^p$ and then yield that $p_{12}=0$ and $[\Phi]=p_{11}I$ where $I$ is the identity matrix since $\Psi$ is an isomorphism. Similarly, we have $\xi^j=\xi^q$, $k=r$, $\iota=\kappa$ and then the claim follows.
\end{proof}

\begin{rmk}
For a left $\D$-module $V$, there exists a left dual module denoted by $V\As$ with module structure given by $(h\rightharpoonup f)(v)=f(S(h)\cdot v)$ for all $h\in \D, v\in V, f\in V\As$. We claim that $V_{i,j,k,\iota}\As\cong V_{-i-1,-j,k+1,\iota+1}$ for all $(i,j,k,\iota)\in \Lambda$. Indeed, denote by $\{v_1,v_2\}$ and $\{f_1,f_2\}$ the basises of $V_{i,j,k,\iota}$ and $V_{i,j,k,\iota}\As$, and denote by $\{w_1, w_2\}$ the basis of $V_{-i-1,-j,k+1,\iota+1}$. After a direct computation, one can show that the map $\varphi: V_{-i-1,-j,k+1,\iota+1}\mapsto V_{i,j,k,\iota}\As $ given by
\begin{align*}
\varphi(w_1)=\xi\Lam_1^2\Lam_4f_2,\quad \varphi(w_2)=f_1,
\end{align*}
is a left $\D$-module isomorphism.
\end{rmk}

Now we show that if $V$ is a $\D$-simple module, then $\dim V=1,2$.
\begin{lem}
For any simple $\D$-module $V$, $\dim V<3$.
\end{lem}
\begin{proof}
Note that by Proposition $\ref{proDouble}$, ${}_{\D(G)}\mathcal{M}\simeq {}_G^G\mathcal{YD}$ and ${}_{\D}\mathcal{M}\simeq \HYD$ as categories and by \cite[Proposition\;2.2.1]{AG99}, $\HYD \simeq {}_A^A\mathcal{YD}$. By Remark $\ref{rmkGAcocycledefor}$, we know that $A\cong G^{\sigma}$ for some $\sigma\in \mathcal{Z}^2(G,\mathds{k})$. Then we have that ${}_A^A\mathcal{YD}\simeq {}_G^G\mathcal{YD}$ by Proposition $\ref{proCocycledeformation}$. Thus the claim holds in ${}_{\D}\mathcal{M}$ if and only if it holds in ${}_{\D(G)}\mathcal{M}$.
From the proof in \cite[Theorem\,2.5]{B03}, the Drinfeld double $\text{gr}\;\D(G)\cong \BN(U)\,\sharp\, \K[\widehat{\Gamma}\times\Gamma]$, where $U=\K\{x,y\}\in{}_{\widehat{\Gamma}\times\Gamma}^{\widehat{\Gamma}\times\Gamma}\mathcal{YD}$ with appropriate actions and coactions and $\BN(U)=T(U)/J$ is a quantum plane where $J$ is the ideal generated by $x^2=0=y^2, xy+yx=0$. Moreover, $\D(G)$ can be regarded as the lifting of the quantum plane $\BN(U)$ by deforming the linking relation $xy+yx=0$. Thanks to \cite[Theorem\,3.5]{AB04}, we can get that if $V$ is a simple $\D(G)$-module, $\dim V<3$.
\end{proof}

Finally, we describe all the simple modules of $\D$ up to isomorphism.
\begin{thm}\label{thmsimplemoduleD}
There exist $48$ simple left $D$-modules pairwise non-isomorphic, among which $16$ one-dimensional modules are given in Lemma $\ref{lemOnesimpleD}$ and $32$ two-dimensional simple modules are given in Lemma $\ref{lemTwosimpleD}$.
\end{thm}

\section{Yetter-Drinfeld modules category $\HYD$}\label{secHYD}
In this section, we translate the simple $\D$-modules to the objects in $\HYD$ following the fact that the category ${}_{\D}\mathcal{M}$ is equivalent to the category $\HYD$. In order to study the Nichols algebras over the simple objects in $\HYD$, We describe explicitly the simple objects in $\HYD$ and also their braidings.
\begin{pro}
Let $\K_{\chi_{i,j,k}}=\K v$ be a one dimensional $\D$-module with $0\leq i,j<2,0\leq k<4$. Then $\K_{\chi_{i,j,k}}\in\HYD$ with the module structure and comodule structure given by
\begin{align*}
a\cdot v&=\xi^k v,\, b\cdot v=0, \, c\cdot v=0,\, d\cdot v=(-1)^i(-1)^j\xi^k v;\\
\delta(v)&=\begin{cases}
               a^{2i}\otimes v  &\text{~if~} j=0,\\
               da^{2i+3}\otimes v & \text{~if~} j=1.
        \end{cases}
\end{align*}
\end{pro}
\begin{proof}
Since $\K_{\chi_{i,j,k}}=\K v$ is a one-dimensional $\D$-module with $0\leq i,j<2,0\leq k<4$, the $H$-action shall be given by the restriction of the character of $\D$ given by Lemma $\ref{lemOnesimpleD}$ and the coaction must be of the form $\delta(v)=t\otimes v$, where $t\in G(H)=\{1,a^2,da,da^3\}$ such that $\langle g, t\rangle v=(-1)^iv$ and $\langle h, t\rangle v=(-1)^jv$. Then the claim follows.
\end{proof}

The following Proposition gives the braiding of $\K_{\chi_{i,j,k}}=\K v$ for all $0\leq i,j<2$, $0\leq k<4$.
\begin{pro}\label{braidingone}
The braiding of the one-dimensional YD-module $\K_{\chi_{i,j,k}}=\K v$ is
\begin{align*}
c(v\otimes v)=\begin{cases}
                  (-1)^{ik}v\otimes v, &\text{~if~} j=0;\\
                  -(-1)^{(i+1)k}v\otimes v, & \text{~if~} j=1.
              \end{cases}
\end{align*}
\end{pro}

\begin{lem}
Let $V_{i,j,k,\iota}=\K\{v_1,v_2\}$ be a two-dimensional simple $\D$-module with $(i,j,k,\iota)\in\Lambda$. If we denote $\Lam_1=\xi^i$, $\Lam_2=\xi^j$, $\Lam_3=(-1)^k$ and $\Lam_4=(-1)^\iota$, then $V_{i,j,k,\iota}\in\HYD$ with the module structure given by
\begin{align*}
a\cdot v_1&=-\Lam_4\Lam_1v_1,\quad b\cdot v_1=0,\quad c\cdot v_1=0,\quad d\cdot v_1=\Lam_1v_1,\\
a\cdot v_2&=\xi\Lam_4\Lam_1 v_2,\quad b\cdot v_2=\Lam_4v_1,\quad c\cdot v_2=v_1,\quad d\cdot v_2=\xi\Lam_1v_2,
\end{align*}
and the comodule structure given by
\begin{itemize}
  \item for $k=0$, i.e., $\Lam_3=1:$
  \begin{align*}
  \delta(v_1)&=a^j\otimes v_1+\xi(\Lam_1\Lam_4+\Lam_1\Lam_2)ba^{j-1}\otimes v_2,\\
  \delta(v_2)&=da^{j-1}\otimes v_2+\frac{1}{2}\xi(\Lam_1^3\Lam_2-\Lam_1^3\Lam_4)ca^{j-1}\otimes v_1,
  \end{align*}
  \item for $k=1$, i.e., $\Lam_3=-1:$
  \begin{align*}
  \delta(v_1)&=da^{j-1}\otimes v_1+\xi(\Lam_1\Lam_4-\Lam_1\Lam_2)ca^{j-1}\otimes v_2,\\
  \delta(v_2)&=a^{j}\otimes v_2-\frac{1}{2}\xi(\Lam_1^3\Lam_2+\Lam_1^3\Lam_4)ba^{j-1}\otimes v_1.
  \end{align*}
\end{itemize}
\end{lem}
\begin{proof}
Note that by Lemma $\ref{lemAtoHdual}$ and Remark $\ref{rmkAtoHdual}$, we have that
\begin{align*}
(g^l)^\ast&=\frac{1}{8}\sum_{i=0}^3 \xi^{-il}a^i+\xi^{-(i+1)l}da^i ,\quad
(g^lh)^\ast=\frac{1}{8} \sum_{i=0}^3 \xi^{-il}a^i -\xi^{-(i+1)l} da^i ,\\
(g^lx)^\ast&=\frac{1}{8\sqrt 2} \sum_{i=0}^3 \xi^{-(i+1)l}ba^i  + \xi^{-(i+1)l} ca^i ,\\
(g^lhx)^\ast&=\frac{1}{8\sqrt 2} \sum_{i=0}^3 \xi^{-(i+1)l}ba^i -\xi^{-(i+1)l}ca^i .
\end{align*}
Denote by $\{h_i\}_{1\leq i\leq 16}$ and $\{h^i\}_{1\leq i\leq 12}$ a basis of $H$ and its dual basis respectively. Then the comodule structure is given by $\delta(v)=\sum_{i=1}^{16}c_i\otimes c^i\cdot v$ for any $v\in V_{i,j,k}$. Thus
\begin{align*}
\delta(v_1)&=\sum_{l=0}^{3}\sum_{n=0}^1(g^lh^n)\As\otimes g^lh^n\cdot v_1+(g^lh^nx)\As\otimes g^lh^nx\cdot v_1\\
           &=\sum_{l=0}^{3}\sum_{n=0}^1\Lam_3^n\Lam_2^l(g^lh^n)\As\otimes v_1
           +\Lam_3^n(\Lam_2)^l(g^lh^nx)\As\otimes  x_2v_2\\
           &=\frac{1}{2}[(1+\Lam_3)a^j+(1-\Lam_3)da^{j-1}]\otimes v_1+\frac{1}{2\sqrt 2}x_2(1+\Lam_3)ba^{j-1}\otimes v_2\\&\quad
           +\frac{1}{2\sqrt 2}x_2(1-\Lam_3)ca^{j-1}\otimes v_2,\\
\delta(v_2)&=\sum_{l=0}^{3}\sum_{n=0}^1((g^l)(g^lh^n)\As\otimes g^lh^n\cdot v_2+(g^lh^nx)\As\otimes g^lh^nx\cdot v_2\\
           &=\sum_{l=0}^{3}\sum_{n=0}^1(-\Lam_3)^n\Lam_2^l(g^lh^n)\As\otimes v_2+(-\Lam_3)^n(\Lam_2)^l(g^lh^nx)\As\otimes  x_1v_1\\
           &=\frac{1}{2}[(1-\Lam_3)a^j+(1+\Lam_3)da^{j-1}]\otimes v_2+\frac{1}{2\sqrt 2}x_1(1-\Lam_3)ba^{j-1}\otimes v_1\\&\quad
           +\frac{1}{2\sqrt 2}x_1(1+\Lam_3)ca^{j-1}\otimes v_1,
\end{align*}
where $x_1=\frac{\sqrt 2}{2}\xi \Lam_1^3(\Lam_2 \Lam_3 -\Lam_4)$ and $x_2=\sqrt 2\xi \Lam_1(\Lam_2\Lam_3+\Lam_4)$. Now we describe the coaction explicitly case by case.

If $\Lam_3=1$, then $x_1=\frac{\sqrt 2}{2}\xi(\Lam_1^3\Lam_2-\Lam_1^3\Lam_4)$, $x_2=\sqrt 2\xi(\Lam_1\Lam_2+\Lam_1\Lam_4)$. In such a case, we have $\delta(v_1)=a^j\otimes v_1+\xi(\Lam_1\Lam_2+\Lam_1\Lam_4)ba^{j-1}\otimes v_2$ and $\delta(v_2)=da^{j-1}\otimes v_2+\frac{1}{2}\xi(\Lam_1^3\Lam_2-\Lam_1^3\Lam_4)ca^{j-1}\otimes v_1$.

If $\Lam_3=-1$, then $x_1=-\frac{\sqrt 2}{2}\xi(\Lam_1^3\Lam_2+\Lam_1^3\Lam_4)$, $x_2=\sqrt 2\xi(\Lam_1\Lam_4-\Lam_1\Lam_2)$. In such a case, we have $\delta(v_1)=da^{j-1}\otimes v_1+\xi(\Lam_1\Lam_4-\Lam_1\Lam_2)ca^{j-1}\otimes v_2$ and $\delta(v_2)=a^{j}\otimes v_2-\frac{1}{2}\xi(\Lam_1^3\Lam_2+\Lam_1^3\Lam_4)ba^{j-1}\otimes v_1$.
\end{proof}

After a direct computation using the braiding in $\HYD$, we describe the braidings of the simple modules $V_{i,j,k,\iota}\in\HYD$.
\begin{pro}\label{probraidsimpletwo}
Let $V_{i,j,k,\iota}=\K\{v_1,v_2\}$ be a two-dimensional simple $\D$-module with $(i,j,k,\iota)\in\Lambda$. If we denote $\Lam_1=\xi^i$, $\Lam_2=\xi^j$, $\Lam_3=(-1)^k$ and $\Lam_4=(-1)^\iota$, then $V_{i,j,k,\iota}\in\HYD$. The braiding of $V_{i,j,k,\iota}=\K\{v_1,v_2\}$ is given by
\begin{itemize}
  \item If $k=0$, then $c(\left[\begin{array}{ccc} v_1\\v_2\end{array}\right]\otimes\left[\begin{array}{ccc} v_1~v_2\end{array}\right])=$
  \begin{align*}
  \left[\begin{array}{ccc}
  -\Lam_4\Lam_1^jv_1\otimes v_1    & \Lam_4(\xi\Lam_1)^jv_2\otimes v_1+[(\xi\Lam_1)^j-\Lam_4\Lam_1^j]v_1\otimes v_2\\
  \Lam_1^jv_1\otimes v_2 &(\xi\Lam_1)^jv_2\otimes v_2+\frac{1}{2}(\xi\Lam_1)^j\Lam_1^2(\Lam_2-\Lam_4)v_1\otimes v_1
  \end{array}\right].
  \end{align*}
  \item If $k=1$, then $ c(\left[\begin{array}{ccc} v_1\\v_2\end{array}\right]\otimes\left[\begin{array}{ccc} v_1~v_2\end{array}\right])=$
  \begin{align*}
  \left[\begin{array}{ccc}
    \Lam_1^jv_1\otimes v_1   & (\xi\Lam_1)^jv_2\otimes v_1+[\Lam_4(\xi\Lam_1)^j+\Lam_1^j]v_1\otimes v_2\\
  -\Lam_4\Lam_1^jv_1\otimes v_2
  &\Lam_4(\xi\Lam_1)^jv_2\otimes v_2-\frac{1}{2}(\xi\Lam_1)^j\Lam_1^2(\Lam_2\Lam_4+1)v_1\otimes v_1
         \end{array}\right].
  \end{align*}
\end{itemize}
\end{pro}
\section{Nichols algebra in $\HYD$}\label{secNicholsalgebaH}
In this section, we determine all finite-dimensional Nichols algebras over simple objects in $\HYD$ and present them by generators and relations.

First, we study the Nichols algebras over one-dimensional simple modules. For convenience, let
\begin{align*}
\Lambda^0=\{(1,0,1),(1,0,3),(0,1,0),(0,1,2),(1,1,0),(1,1,1),(1,1,2),(1,1,3)\}.
\end{align*}
By Proposition $\ref{braidingone}$, the following result follows immediately.
\begin{lem}\label{lemNicholsgeneratedbyone}
The Nichols algebra $\BN(\K_{\chi_{i,j,k}})$ over $\K_{\chi_{i,j,k}}=\K v$ with $0\leq i,j<2,0\leq k<4$ is
\begin{align*}
\BN(\K_{\chi_{i,j,k}})=\begin{cases}
\bigwedge \K_{\chi_{i,j,k}} & (i,j,k)\in\Lambda^0,\\
\K[v] & \text{others}.
\end{cases}
\end{align*}
Moreover, let $V=\oplus_{\ell\in I}V_{\ell}$, where $V_{\ell}\cong \K_{\chi_{i,j,k}}$ for some $(i,j,k)\in\Lambda^0$, and $I$ is a finite index set. Then $\BN(V)=\bigwedge V\cong \otimes_{i\in I}\BN(V_i)$.
\end{lem}

Next, we analyze the Nichols algebras associated to two-dimensional simple modules. For convenience, denote by $\Lambda^i$ with $1\leq i\leq 6$ the finite subsets of $\Lambda$ given by
\begin{align*}
\Lambda^1&=\{(2,1,0,0),(0,1,1,0),(2,3,0,0),(0,3,1,0),\\&\quad~~(0,1,0,1),(0,1,1,1),(0,3,0,1) (0,3,1,1)\},\\
\Lambda^2&=\{(1,3,1,1),(3,3,0,1),(1,1,1,1),(3,1,0,1),\\&\quad~~(3,3,1,0),(3,3,0,0),(3,1,1,0),(3,1,0,0)\},\\
\Lambda^3&=\{(0,1,0,0),(0,3,0,0),(2,1,0,1),(2,3,0,1)\},\\
\Lambda^4&=\{(2,1,1,0),(2,1,1,1),(2,3,1,0),(2,3,1,1)\},\\
\Lambda^5&=\{(1,3,0,1),(1,3,0,0),(1,1,0,1),(1,1,0,0)\},\\
\Lambda^6&=\{(3,3,1,1),(1,3,1,0),(3,1,1,1),(1,1,1,0)\}.
\end{align*}
It is clear that $\Lambda=\cup_{i=1}^6\Lambda_i$.
\begin{lem}
 $\dim\BN(V_{i,j,k,\iota})=\infty$ for all $(i,j,k,\iota)\in\Lambda^1\cup \Lambda^2$.
\end{lem}
\begin{proof}
By Proposition $\ref{probraidsimpletwo}$, the braiding of $V_{i,j,k,\iota}$ for $(i,j,k,\iota)\in\Lambda^1$
has an eigenvector of eigenvalue $1$. Indeed, $c(v_1\otimes v_1)=v_1\otimes v_1$ in the above cases. And for any element $(i,j,k,\iota)\in\Lambda^2$, there exists one element $(p,q,r,\mu)\in\Lambda^2$ such that $V_{i,j,k,\iota}\As\cong V_{p,q,r,\mu}$ in $\HYD$, then by Proposition $\ref{proNicholsdual}$, the claim follows.
\end{proof}

 Then we will show that Nichols algebras $\BN(V_{i,j,k,\iota})$ are finite-dimensional for all $(i,j,k,\iota)\in \Lambda^i$ for $3\leq i\leq 6$, and describe them in terms of generators and relations.

\begin{pro}\label{proV3}
$\BN(V_{i,j,k,\iota}):=\K\langle v_1, v_2\mid v_1^2=0, v_1v_2-\xi^jv_2v_1=0, v_2^4=0\rangle$ for $(i,j,k,\iota)\in\Lambda^3$. In particular, $\dim\BN(V_{i,j,k,\iota})=8$ for $(i,j,k,\iota)\in\Lambda^3$.
\end{pro}
\begin{proof}
In this case, note that $\Lam_1=\Lam_4=\Lam_1^j$, $\delta(v_1)=a^j\otimes v_1+\xi(1+\xi^j\Lam_1)ba^{j-1}\otimes v_2$, $\delta(v_2)=da^{j-1}\otimes v_2+\frac{1}{2}\xi\Lam_1^3(\xi^j-\Lam_4)ca^{j-1}\otimes v_1$, and the braiding is given by
\begin{align*}
   c(\left[\begin{array}{ccc} v_1\\v_2\end{array}\right]\otimes\left[\begin{array}{ccc} v_1~v_2\end{array}\right])=
   \left[\begin{array}{ccc}
       -v_1\otimes v_1    & \xi^jv_2\otimes v_1+(\xi^j\Lam_1^j-1)v_1\otimes v_2\\
       \Lam_1^jv_1\otimes v_2            & \xi^j\Lam_1^jv_2\otimes v_2+\frac{1}{2}\xi^j\Lam_1^j(\xi^j-\Lam_4)v_1\otimes v_1
         \end{array}\right].
\end{align*}
Using the braiding, we have that
\begin{align*}
\Delta(v_1^2)&=v_1^2\otimes 1+v_1^2\otimes 1,\\
\Delta(v_1v_2)&=v_1v_2\otimes 1+\xi^j\Lam_1^jv_1\otimes v_2+\xi^jv_2\otimes v_1+1\otimes v_1v_2,\\
\Delta(v_2v_1)&=v_2v_1\otimes 1+ \Lam_1^jv_1\otimes v_2+v_2\otimes v_1+1\otimes v_2v_1,\\
\Delta(v_2^2)&=v_2^2\otimes 1+(1+\xi^j\Lam_1^j)v_2\otimes v_2-\frac{1}{2}(\Lam_1^j+\xi^j)v_1\otimes v_1+1\otimes v_2^2,
\end{align*}
which give us the relations $x^2=0$, and $v_1v_2-\xi^jv_2v_1=0$. And since
\begin{align*}
c(v_2\otimes v_2^2)&=da^j\cdot v_2^2\otimes v_2+\frac{1}{2}\xi\Lam_1^3(\xi^j-\Lam_4)ca^{j-1}\cdot v_2^2\otimes v_1\\
&=(da^{j-1}\cdot v_2)(da^{j-1}\cdot v_2)\otimes v_2+(ca^{j-1}\cdot v_2)(ba^{j-1}\cdot v_2)\otimes v_2
\\&\ +\frac{1}{2}\xi\Lam_1^3(\xi^j{-}\Lam_4)[(ca^{j-1}\cdot v_2)(a^j\cdot v_2)+(da^{j-1}\cdot v_2)(ca^{j-1}\cdot v_2)]\otimes v_1\\
&=-v_2^2\otimes v_2+\Lam_4v_1^2\otimes v_2+\frac{1}{2}(1{-}\xi^j\Lam_1^3)(\Lam_1v_2v_1+v_1v_2)\otimes v_1\\
&=-v_2^2\otimes v_2+\Lam_1v_2v_1\otimes v_1,
\end{align*}
we have
\begin{align*}
\Delta(v_2^3)&=(v_2\otimes 1+1\otimes v_2)\Delta(v_2^2)\\
&=v_2^3\otimes 1+(1+\xi^j\Lam_1^j)v_2^2\otimes v_2-\frac{1}{2}(\Lam_1^j+\xi^j)v_2v_1\otimes v_1+v_2\otimes v_2^2-v_2^2\otimes v_2
\\&\quad+\Lam_1v_2v_1\otimes v_1+(\xi^j\Lam_1^j{-}1)v_2\otimes v_2^2-\xi^jv_1\otimes v_1v_2-\frac{1}{2}(1{+}\xi^j\Lam_1^j)v_1\otimes v_2v_1\\&\quad+1\otimes v_2^3\\
&=v_2^3\otimes 1+1\otimes v_2^3+\xi^j\Lam_1^jv_2^2\otimes v_2+\frac{1}{2}(\Lam_1^j-\xi^j)v_2v_1\otimes v_1+\xi^j\Lam_1^jv_2\otimes v_2^2\\&\quad+\frac{1}{2}(1{-}\xi^j\Lam_1^j)v_1\otimes v_2v_1.
\end{align*}
Similarly, after a direct computation, we have that
\begin{align*}
\Delta(v_2^4)=(v_2\otimes 1+v_2\otimes v_1)\Delta(v_2^3)=v_2^4\otimes 1+1\otimes v_2^4,
\end{align*}
since $c(v_2\otimes v_2v_1)=\xi^jv_2v_1\otimes v_2$, and $c(v_2\otimes v_2^3)=\frac{1}{2}(\xi^j-\Lam_4)v_1\otimes v_2^2v_1-\xi^j\Lam_1v_2^3\otimes v_2$. Thus, we get relation $x_2^4=0$.

Thus there exists a graded Hopf algebra epimorphism $\pi:R=T(V_{i,j,k,\iota})/I\twoheadrightarrow \BN(V_{i,j,k,\iota})$ in $\HYD$, where $I$ is the ideal generated by the relations $v_1^2=0, v_1v_2-\xi^jv_2v_1=0, v_2^4=0$.
Note that $R(5)=0, R(1)=V_{3,1}, R(0)=\K$ and $R(4)\neq 0$. Then by the Poincar\'{e} duality, we have that $\dim R(4)=\dim R(0)=1$, and $\dim R(3)=\dim R(1)=2$. And it is clear that $R(2)=\K\{v_2^2,v_1v_2\}$ and whence $\dim R(2)=2$. Since $\dim \BN^5(V_{i,j,k,\iota})=0$ and $\pi$ is injective in degree 0 and 1, we have $\dim R(4)=\dim\BN^4(V_{i,j,k,\iota})$ and $\dim R(3)=\dim\BN^3(V_{i,j,k,\iota})$. Moreover, it is clear that $\dim \BN^2(V_{i,j,k,\iota})=2$ which implies $\dim R=\dim\BN(V_{i,j,k,\iota})$. Then the claim follows.
\end{proof}

\begin{pro}\label{proV4}
$\BN(V_{i,j,k,\iota}):=\K\langle v_1, v_2\mid v_1^2=0, v_1v_2+\xi^jv_2v_1=0, v_2^4=0\rangle$ for $(i,j,k,\iota)\in\Lambda^4$. In particular, $\dim\BN(V_{i,j,k,\iota})=8$ for $(i,j,k,\iota)\in\Lambda^4$.
\end{pro}
\begin{proof}
In such a case, note that $\Lam_1=-1=\Lam_3$,
$\delta(v_1)=da^{j-1}\otimes v_1+\xi(\Lam_2-\Lam_4)ca^{j-1}\otimes v_2$,
$\delta(v_2)=a^{j}\otimes v_2+\frac{1}{2}\xi(\Lam_2+\Lam_4)ba^{j-1}\otimes v_1$. And the braiding is given by
\begin{align*}
 c(\left[\begin{array}{ccc} v_1\\v_2\end{array}\right]\otimes\left[\begin{array}{ccc} v_1~v_2\end{array}\right])=
  \left[\begin{array}{ccc}
    -v_1\otimes v_1   & -\xi^jv_2\otimes v_1-(\Lam_4\xi^j+1)v_1\otimes v_2\\
    \Lam_4v_1\otimes v_2
  &-\Lam_4\xi^jv_2\otimes v_2+\frac{1}{2}\xi^j(\Lam_2\Lam_4+1)v_1\otimes v_1
         \end{array}\right].
\end{align*}
Using the braiding, we have that
\begin{align*}
\Delta(v_1^2)&=v_1^2\otimes 1+1\otimes v_1^2,\\
\Delta(v_1v_2)&=v_1v_2\otimes 1-\xi^jv_2\otimes v_1-\Lam_4\xi^jv_1\otimes v_2+1\otimes v_1v_2,\\
\Delta(v_2v_1)&=v_2v_1\otimes 1+v_2\otimes v_1+\Lam_4v_1\otimes v_2+1\otimes v_2v_1,\\
\Delta(v_2^2)&=v_2^2\otimes 1+(1-\Lam_4\xi^j)v_2\otimes v_2+\frac{1}{2}(\xi^j-\Lam_4)v_1\otimes v_1+1\otimes v_2^2,
\end{align*}
 which gives us the relations $v_1^2=0$ and $v_1v_2+\xi^jv_2v_1=0$. Since $c(v_2\otimes v_2^2)=-v_2^2\otimes v_2+\Lam_4v_2v_1\otimes v_1$, we have that
 \begin{align*}
 \Delta(v_2^3)&=v_2^3\otimes 1-\Lam_4\xi^jv_2^2\otimes v_2-\Lam_4\xi^jv_2\otimes  v_2^2 +\frac{1}{2}(\xi^j+\Lam_4)v_2v_1\otimes v_1\\&\quad+\frac{1}{2}(\Lam_2\Lam_4+1)v_1\otimes v_2v_1+1\otimes v_2^3.
 \end{align*}
 Similarly, after a direct computation, we also have that
 \begin{align*}
 \Delta(v_2^4)=(v_2\otimes 1+1\otimes v_2)\Delta(v_2^3)=v_2^4\otimes 1+1\otimes v_2^4,
 \end{align*}
 since
 \begin{align*}
 c(v_2\otimes v_2v_1)&=-\xi^jv_2v_1\otimes v_2-\frac{1}{2}\xi^j(\Lam_2+\Lam_4)v_1^2\otimes v_1=-\xi^jv_2v_1\otimes v_2,\\
 c(v_2\otimes v_2^3)&=\Lam_4\xi^jv_2^3\otimes v_2-\frac{1}{2}(\xi^j+\Lam_4)v_2^2v_1\otimes v_1.
 \end{align*}
 This gives us relation $v_2^4=0$.
\end{proof}

\begin{pro}\label{proV5}
$\BN(V_{i,j,k,\iota}):=\K\langle v_1, v_2\mid v_1^4=0, v_1v_2+\Lam_4v_2v_1=0, v_1^2+2\Lam_4v_2^2\rangle$ for $(i,j,k,\iota)\in\Lambda^5$. In particular, $\dim\BN(V_{i,j,k,\iota})=8$ for $(i,j,k,\iota)\in\Lambda^5$.
\end{pro}
\begin{proof}
In such a case, note that $\Lam_1=\xi$,
$\delta(v_1)=a^{j}\otimes v_1-(\Lam_2+\Lam_4)ba^{j-1}\otimes v_2$,
$\delta(v_2)=da^{j-1}\otimes v_2+\frac{1}{2}(\Lam_2-\Lam_4)ca^{j-1}\otimes v_1$. And the braiding is given by
\begin{align*}
 c(\left[\begin{array}{ccc} v_1\\v_2\end{array}\right]\otimes\left[\begin{array}{ccc} v_1~v_2\end{array}\right])=
  \left[\begin{array}{ccc}
    -\Lam_4\xi^jv_1\otimes v_1   & -\Lam_4v_2\otimes v_1-(\Lam_4\xi^j+1)v_1\otimes v_2\\
    \xi^jv_1\otimes v_2      &-v_2\otimes v_2+\frac{1}{2}(\Lam_2-\Lam_4)v_1\otimes v_1
  \end{array}\right].
\end{align*}
Using the braiding, we have that
\begin{align*}
\Delta(v_1^2)&=v_1^2\otimes 1+(1-\Lam_4\xi^j)v_1\otimes v_1+1\otimes v_1^2,\\
\Delta(v_1v_2)&=v_1v_2\otimes 1-\Lam_4v_2\otimes v_1-\Lam_4\xi^jv_1\otimes v_2+1\otimes v_1v_2,\\
\Delta(v_2v_1)&=v_2v_1\otimes 1+v_2\otimes v_1+\xi^jv_1\otimes v_2+1\otimes v_2v_1,\\
\Delta(v_2^2)&=v_2^2\otimes 1+\frac{1}{2}(\xi^j-\Lam_4)v_1\otimes v_1+1\otimes v_2^2,
\end{align*}
 which gives us the relations $v_1^2+2\Lam_4v_2^2=0$ and $v_1v_2+\Lam_4v_2v_1=0$. Since $c(v_1\otimes v_1^2)=-v_1^2\otimes v_1$ and $c(v_1\otimes v_1^3)=\xi^j\Lam_4v_1^3\otimes v_1$, we have that
 \begin{align*}
 \Delta(v_1^3)&=v_1^3\otimes 1+1\otimes v_1^3-\Lam_4\xi^jv_1^2\otimes v_1-\Lam_4\xi^jv_1\otimes v_1^2,\\
 \Delta(v_1^4)&=(v_1\otimes 1+1\otimes v_1)\Delta(v_1^3)=v_1^4\otimes 1 +1\otimes v_1^4,
 \end{align*}
 which implies that $v_1^4=0$.
\end{proof}

\begin{pro}\label{proV6}
$\BN(V_{i,j,k,\iota}):=\K\langle v_1, v_2\mid v_1^4=0, v_1v_2+\Lam_4v_2v_1=0, v_1^2+2\Lam_4v_2^2=0\rangle$ for $(i,j,k,\iota)\in\Lambda^6$. In particular, $\dim\BN(V_{i,j,k,\iota})=8$ for $(i,j,k,\iota)=(i,j,k,\iota)\in\Lambda^6$.
\end{pro}
\begin{proof}
In such a case, note that $\Lam_1\Lam_4=\xi$, $\Lam_1^2=-1$ and $\Lam_1\xi=-\Lam_4$.
$\delta(v_1)=da^{j-1}\otimes v_1+\xi\Lam_1(\Lam_4-\Lam_2)ca^{j-1}\otimes v_2$,
$\delta(v_2)=a^{j}\otimes v_2-\frac{1}{2}\xi\Lam_1^3(\Lam_2+\Lam_4)ba^{j-1}\otimes v_1$. And the braiding is given by
\begin{align*}
 c(\left[\begin{array}{ccc} v_1\\v_2\end{array}\right]\otimes\left[\begin{array}{ccc} v_1~v_2\end{array}\right])=
  \left[\begin{array}{ccc}
    \Lam_1^jv_1\otimes v_1   & (\xi\Lam_1)^jv_2\otimes v_1+(\Lam_1^j-1)v_1\otimes v_2\\
    -\xi^jv_1\otimes v_2      &-v_2\otimes v_2+\frac{1}{2}\xi^j(\Lam_1^j-1)v_1\otimes v_1
  \end{array}\right].
\end{align*}
Using the braiding, we have that
\begin{align*}
\Delta(v_1^2)&=v_1^2\otimes 1+(1+\Lam_1^j)v_1\otimes v_1+1\otimes v_1^2,\\
\Delta(v_1v_2)&=v_1v_2\otimes 1+(\xi\Lam_1)^jv_2\otimes v_1+\Lam_1^jv_1\otimes v_2+1\otimes v_1v_2,\\
\Delta(v_2v_1)&=v_2v_1\otimes 1+v_2\otimes v_1-\xi^jv_1\otimes v_2+1\otimes v_2v_1,\\
\Delta(v_2^2)&=v_2^2\otimes 1+\frac{1}{2}\xi^j(\Lam_1^j-1)v_1\otimes v_1+1\otimes v_2^2\\
&=v_2^2\otimes 1+\frac{1}{2}\xi^j\Lam_1^j(1+\Lam_1^j)v_1\otimes v_1+1\otimes v_2^2,
\end{align*}
 which gives us the relations $v_1^2-2\xi^j\Lam_1^j v_2^2=0$ and $v_1v_2-(\Lam_1\xi)^jv_2v_1=0$. Since $c(v_1\otimes v_1^2)=-v_1^2\otimes v_1$ and $c(v_1\otimes v_1^3)=-\Lam_1^jv_1^3\otimes v_1$, we have that
 \begin{align*}
 \Delta(v_1^3)&=v_1^3\otimes 1+1\otimes v_1^3+\Lam_1^jv_1^2\otimes v_1+\Lam_1^jv_1\otimes v_1^2,\\
 \Delta(v_1^4)&=(v_1\otimes 1+1\otimes v_1)\Delta(v_1^3)=v_1^4\otimes 1 +1\otimes v_1^4,
 \end{align*}
 which implies that $v_1^4=0$.
\end{proof}

Finally we show that the Nichols algebra $\BN(V)$ over a simple object $V$ in $\HYD$ is finite-dimensional if and only if $V$ is isomorphic either to $\K_{\chi_{i,j,k}}$ with $(i,j,k)\in\Lambda^0$ or $V_{i,j,k,\iota}$ with $(i,j,k,\iota)\in\cup_{3\leq\ell\leq 6}\Lambda^{\ell}$.
\begin{proofthma}
For any simple object $V\in\HYD$ such that $\dim\BN(V)<\infty$. If $\dim V=1$, then $V\cong \K_{\chi_{i,j,k}}$ where $(i,j,k)\in\Lambda^0$ by Lemma $\ref{lemNicholsgeneratedbyone}$. If $\dim V=2$, then $V\cong V_{i,j,k,\iota}$ where $(i,j,k,\iota)\in \cup_{3\leq\ell\leq 6}\Lambda^{\ell}$ by Propositions $\ref{proV3}$, $\ref{proV4}$, $\ref{proV5}$ and $\ref{proV6}$. And it is clear that these Nichols algebras are pairwise non-isomorphic, since their infinitesimal braidings are pairwise non-isomorphic in $\HYD$.
\end{proofthma}
\section{Hopf algebras over $H$}\label{secHopfalgebraH}
In this section, we determine all finite-dimensional Hopf algebras $A$ over $H$ and the corresponding infinitesimal braidings $V$ are isomorphic to $\K_{\chi_{i,j,k}}$ with $(i,j,k)\in\Lambda^0$, or $V_{i,j,k,\iota}\in\cup_{3\leq \ell\leq 6}\Lambda^{\ell}$.

First, we show that such Hopf algebras mentioned above are generated in degree one with respect to the standard filtration, i.e., $\text{gr}\,A\cong \BN(V)\sharp H$.
\begin{lem}
Let $A$ be a Hopf algebra such that $A_{[0]}\cong H$ and the corresponding infinitesimal braiding $V$ is either the simple modules $\K_{\chi_{i,j,k}}$ with $(i,j,k)\in\Lambda^0$, and $V_{i,j,k,\iota}\in\cup_{3\leq \ell\leq
6}\Lambda^{\ell}$. Then $\text{gr}\,A\cong \BN(V)\sharp H$. That is, $A$ is generated by the first term of the standard filtration.
\end{lem}\label{lemGenerataionindegreeone}
\begin{proof}
Recall that $T=\text{gr}\,A=\oplus_{i\geq 0}A_{[i]}/A_{[i+1]}=R\sharp H$, where $A_{[0]}\cong H$ and $R=T^{coA_{[0]}}$. In order to show that $gr\,A\cong \BN(V)\sharp H$, i.e., $R\cong \BN(V)$, let $S=R\As$ be the graded dual of $R$ and by the duality principle in $\cite[Lemma\;2.4]{AS02}$, $S$ is generated by $S(1)$ since $\Pp(R)=R(1)$. Thus there exists a surjective morphism $S\twoheadrightarrow \BN(W)$ where $W=S(1)$. Thus $S$ is a Nichols algebra if $\Pp(S)=S(1)$, which implies $R$ is a Nichols algebra, i.e., $R=\BN(V)$. To show that $\Pp(S)=S(1)$, it is enough to prove that the relations of $\BN(V)$ also hold in $S$.

Assume $W=\K_{\chi_{i,j,k}}=\K[v]/(v^2)$ with $(i,j,k)\in\Lambda^0$ and then $\BN(W)=\bigwedge \K_{\chi_{i,j,k}}$ for $(i,j,k)\in\Lambda^0$. In such a case, if $v^2\in S$, then $v^2$ is a primitive element and $c(v^2\otimes v^2)=v^2\otimes v^2$. Since as the graded dual of $R$, $S$ must be finite-dimensional, thus $v^2=0$. Then the claim follows.

Assume that $W=V_{i,j,k,\iota}$ with $(i,j,k,\iota)\in\Lambda^3$, then by Proposition $\ref{proV3}$, we know that as an algebra $\BN(W):=\K\langle v_1, v_2|v_1^2=0, v_1v_2-\xi^jv_2v_1=0, v_2^4=0\rangle$ and the relations of $\BN(W)$ are all primitive elements. Thus we need to show that $c(r\otimes r)=r\otimes r$ for $r=v_1^2$, $v_1v_2-\xi^jv_2v_1$ and $v_2^4$. Since
\begin{align*}
\delta(v_1)&=a^j\otimes v_1+\xi(1+\xi^j\Lam_1)ba^{j-1}\otimes v_2, \\
\delta(v_2)&=da^{j-1}\otimes v_2+\frac{1}{2}\xi\Lam_1^3(\xi^j-\Lam_2)ca^{j-1}\otimes v_1,
\end{align*}
after a direct computation, we have that
\begin{gather*}
\delta(v_1^2)=a^2\otimes v_1^2+\xi(\Lam_2-\Lam_1)ba\otimes(v_1v_2-\xi^jv_2v_1),\\
\delta(v_1v_2-\xi^jv_2v_1)=da\otimes (v_1v_2-\xi^jv_2v_1),\quad
\delta(v_2^4)=1\otimes v_2^3.
\end{gather*}
Thus by the braiding of $\HYD$, we have
\begin{gather*}
c(v_1^2\otimes v_1^2)=v_1^2\otimes v_1^2,\quad c(v_2^4\otimes v_2^4)=v_2^4\otimes v_2^4,\\
c((v_1v_2-\xi^jv_2v_1)\otimes (v_1v_2-\xi^2v_2v_1))=(v_1v_2-\xi^2v_2v_1)\otimes (v_1v_2-\xi^2v_2v_1).
\end{gather*}
Then the claim follows.
Similarly, the claim follows when $W=V_{i,j,k,\iota}$ for $(i,j,k,\iota)\in\Lambda^4\cup\Lambda^5\cup\Lambda^6$.
\end{proof}

Next, we shall show that there do not exist non-trivial liftings for the bosonizations of the Nichols algebras over $\K_{\chi_{i,j,k}}$ with $(i,j,k)\in\Lambda^0$, and $V_{i,j,k,\iota}$ with $(i,j,k,\iota)\in\Lambda^3\cup\Lambda^4$.
\begin{pro}\label{proLiftingnon0}
Let $A$ be a finite-dimensional Hopf algebra over $H$ such that its infinitesimal braiding $V$ is isomorphic to $\K_{\chi_{i,j,k}}$ with $(i,j,k)\in\Lambda^0$. Then $A\cong \bigwedge \K_{\chi_{i,j,k}}\sharp H$.
\end{pro}
\begin{proof}
Note that $\text{gr}\, A\cong \BN(V)\sharp H$, where $V$ is isomorphic to $\K_{\chi_{i,j,k}}$ with $(i,j,k)\in\Lambda^0$. We prove that the relations in $\BN(V)$ also hold in $H$. Indeed, let $\bigwedge \K_{\chi_{i,j,k}}=\K[v]/(v^2)$,
If $j=0$, then $\delta(v)=a^2\otimes v$. In such a case, we have that
\begin{align*}
\Delta_A(v)=v\otimes 1+a^2\otimes v,\quad
\Delta_A(v^2)=v^2\otimes 1+ 1\otimes v^2.
\end{align*}
If $j=1$, then $\delta(v)=da^{2i+3}\otimes v$. In such a case, we have that
\begin{align*}
\Delta(v)=v\otimes 1+da^{2i+3}\otimes v,\quad
\Delta(v^2)=v^2\otimes 1 +1\otimes v^2.
\end{align*}
But since $A$ is a finite-dimensional Hopf algebra so that $A$ cannot contain any primitive element. Therefore the relation $v^2=0$ must hold in $A$.
\end{proof}

\begin{pro}\label{proLiftingnon3}
Let $A$ be a finite-dimensional Hopf algebra over $H$ such that the infinitesimal braiding $V$ is isomorphic either to $V_{i,j,k,\iota}$ where $(i,j,k,\iota)\in \Lambda^3$. Then $A\cong \BN(V_{i,j,k,\iota})\sharp H$ where $(i,j,k,\iota)\in \Lambda^3$.
\end{pro}
\begin{proof}
Note that $\text{gr}\, A\cong \BN(V_{i,j,k,\iota})\sharp H$, where $(i,j,k,\iota)\in \Lambda^3$. For convenience, let $W:=V_{i,j,k,\iota}$ where $(i,j,k,\iota)\in \Lambda^3$.
In such a case, the bosonization $\BN(W)\sharp H$ is generated by $x, y, a, b,c,d$ satisfying the relations
\begin{align*}
a^4=1,\quad b^2=0,\quad c^2=0, \quad d^4=1,\quad a^2d^2=1,\quad ad=da,\quad bc=0=cb,\\
ab=\xi ba,\quad ac=\xi ca,\quad bd=\xi db,\quad cd=\xi dc,\quad bd=ca,\quad
ba=cd,\quad ax=-xa,\\ bx=-xb,\quad cx=\Lam_1xc,\quad dx=\Lam_1xd,\quad xy-\xi^jyx=0,\quad x^2=0,\quad y^4=0,\\
ay-\xi ya=\Lam_4xc,\quad by-\xi yb=\Lam_4xd,\quad
cy-\xi\Lam_1yc=xa,\quad dy-\xi\Lam_1 yd=xb.
\end{align*}
The coalgebra structure is given by
\begin{align*}
\Delta(a)&=a\otimes a+b\otimes c,\quad \Delta(b)=a\otimes b+b\otimes d,\\
\Delta(c)&=c\otimes a+d\otimes c,\quad \Delta(d)=d\otimes d+c\otimes b,\\
\Delta(x)&=x\otimes 1+a^j\otimes x+\xi(1+\xi^j\Lam_1)ba^{j-1}\otimes y,\\
\Delta(y)&=y\otimes 1+da^{j-1}\otimes y+\frac{1}{2}\xi\Lam_1^3(\xi^j-\Lam_4)ca^{j-1}\otimes x.
\end{align*}
Assume that $A$ is a finite-dimensional Hopf algebra such that $\text{gr}\,A\cong \BN(W)\sharp H$. After a direct computation, we have that
\begin{gather*}
\Delta(x^2)=x^2\otimes 1+a^2\otimes x^2+\xi(\xi^j-\Lam_1)ba\otimes (xy-\xi^jyx),\\
\Delta(xy-\xi^jyx)=(xy-\xi^jyx)\otimes 1+da\otimes(xy-\xi^jyx).
\end{gather*}
From the second equation, we have that
\begin{align*}
xy-\xi^jyx\in\Pp_{1,da}(\BN(W)\sharp H)=\Pp_{1,da}(H)=\K\{1-da\},
\end{align*}
which implies that $xy-\xi^jyx=\mu(1-da)$ for some $\mu\in\K$. Then from the first equation, we get that
\begin{align*}
\Delta(x^2+\xi(\xi^j-\Lam_1)\mu ba)&=x^2\otimes 1+a^2\otimes x^2+\xi(\xi^j-\Lam_1)ba\otimes \mu(1-da)\\&\quad\,
+\xi(\xi^j-\Lam_1)\mu ba\otimes da+\xi(\xi^j-\Lam_1)\mu a^2\otimes ba\\
&=(x^2+\xi(\xi^j{-}\Lam_1)\mu ba)\otimes 1+a^2\otimes(x^2+\xi(\xi^j{-}\Lam_1)\mu ba),
\end{align*}
which implies that there exists $\nu\in\K$ such that
\begin{align*}
x^2+\xi(\xi^j-\Lam_1)\mu ba=\nu(1-a^2).
\end{align*}
Since $ax^2=x^2a$, $bx^2=x^2b$ and $ab=\xi ba$, we get that $\mu=0=\nu$ and whence the relations $x^2=0$ and $xy-\xi^jyx=0$ must hold in $A$. Then after a tedious computation, we have that
\begin{align*}
\Delta(y^4)=\Delta(y)^4=y^4\otimes 1+1\otimes y^4,
\end{align*}
which implies that relation $y^4=0$ holds in $A$. Thus the claim follows.
\end{proof}

\begin{pro}\label{proLiftingnon4}
Let $A$ be a finite-dimensional Hopf algebra over $H$ such that the infinitesimal braiding $V$ is isomorphic either to $V_{i,j,k,\iota}$ where $(i,j,k,\iota)\in \Lambda^4$. Then $A\cong \BN(V_{i,j,k,\iota})\sharp H$ where $(i,j,k,\iota)\in \Lambda^4$.
\end{pro}
\begin{proof}
Note that $\text{gr}\, A\cong \BN(V_{i,j,k,\iota})\sharp H$, where $(i,j,k,\iota)\in \Lambda^4$. For convenience, let $W:=V_{i,j,k,\iota}$ where $(i,j,k,\iota)\in \Lambda^4$.
In such a case, the bosonization $\BN(W)\sharp H$ is generated by $x, y, a, b,c,d$ satisfying the relations
\begin{align*}
a^4=1,\quad b^2=0,\quad c^2=0, \quad d^4=1,\quad a^2d^2=1,\quad ad=da,\quad bc=0=cb,\\
ab=\xi ba,\quad ac=\xi ca,\quad bd=\xi db,\quad cd=\xi dc,\quad bd=ca,\quad
ba=cd,\quad ax=\Lam_4xa,\\
 bx=\Lam_4xb,\quad cx=-xc,\quad dx=-xd,\quad xy+\xi^jyx=0,\quad x^2=0,\quad y^4=0,\\
ay+\xi\Lam_4 ya=\Lam_4xc,\quad by+\xi\Lam_4 yb=\Lam_4xd,\quad
cy+\xi yc=xa,\quad dy+\xi yd=xb.
\end{align*}
the coalgebra structure is given by
\begin{align*}
\Delta(a)&=a\otimes a+b\otimes c,\quad \Delta(b)=a\otimes b+b\otimes d,\\
\Delta(c)&=c\otimes a+d\otimes c,\quad \Delta(d)=d\otimes d+c\otimes b,\\
\Delta(x)&=x\otimes 1+da^{j-1}\otimes x+\xi(\Lam_2-\Lam_4)ca^{j-1}\otimes y,\\
\Delta(y)&=y\otimes 1+a^{j}\otimes y+\frac{1}{2}\xi(\Lam_2+\Lam_4)ba^{j-1}\otimes x.
\end{align*}
Assume that $A$ is a finite-dimensional Hopf algebra such that $\text{gr}\,A\cong \BN(W)\sharp H$. After a direct computation, we have that
\begin{gather*}
\Delta(x^2)=x^2\otimes 1+a^2\otimes x^2+\xi(1+\xi^j\Lam_4)ba\otimes (xy+\xi^jyx),\\
\Delta(xy+\xi^jyx)=(xy+\xi^jyx)\otimes 1+da\otimes(xy+\xi^jyx).
\end{gather*}
From the second equation, we have that
\begin{align*}
xy+\xi^jyx\in\Pp_{1,da}(\BN(W)\sharp H)=\Pp_{1,da}(H)=\K\{1-da\},
\end{align*}
which implies that $xy+\xi^jyx=\mu(1-da)$ for some $\mu\in\K$. Then from the first equation, we get that
\begin{align*}
\Delta(x^2+\xi(1{+}\xi^j\Lam_4)\mu ba)&=x^2\otimes 1+a^2\otimes x^2+\xi(1+\xi^j\Lam_4)ba\otimes \mu(1-da)\\&\quad\,
+\xi(1+\xi^j\Lam_4)\mu ba\otimes da+\xi(1+\xi^j\Lam_4)\mu a^2\otimes ba\\
&=(x^2{+}\xi(1{+}\xi^j\Lam_4)\mu ba)\otimes 1+a^2\otimes(x^2{+}\xi(1{+}\xi^j\Lam_4)\mu ba),
\end{align*}
which implies that there exists $\nu\in\K$ such that
\begin{align*}
x^2+\xi(1+\xi^j\Lam_4)\mu ba=\nu(1-a^2).
\end{align*}
Since $ax^2=x^2a$, $bx^2=x^2b$ and $ab=\xi ba$, we get that $\mu=0=\nu$ and whence the relations $x^2=0$ and $xy+\xi^jyx=0$ must hold in $A$. Then after a tedious computation, we have that
\begin{align*}
\Delta(y^4)=\Delta(y)^4=y^4\otimes 1+1\otimes y^4,
\end{align*}
which implies that relation $y^4=0$ holds in $A$. Thus the claim follows.
\end{proof}


Now we define two families of Hopf algebras $\LA$ and $\LAA$ and show that they are indeed liftings of the Nichols algebras $\BN(V_{i,j,k,\iota})$ for $(i,j,k,\iota)\in \Lambda^5\cup\Lambda^6$.
\begin{defi}
For $\mu\in\K$ and $(i,j,k,\iota)\in \Lambda^5$, Let $\LA$ be the algebra generated by $x$, $y$, $a$, $b$, $c$, $d$ satisfying the relations
\begin{align*}
a^4=1,\quad b^2=0,\quad c^2=0, \quad d^4=1,\quad a^2d^2=1,\quad ad=da,\quad bc=0=cb,\\
ab=\xi ba,\quad ac=\xi ca,\quad bd=\xi db,\quad cd=\xi dc,\quad bd=ca,\quad
ba=cd.\\
ax=-\Lam_4\xi xa,\quad bx=-\Lam_4\xi xb,\quad cx=\xi xc,\quad dx=\xi xd,\\
ay+\Lam_4 ya=\Lam_4xc,\quad by+\Lam_4 yb=\Lam_4xd,\quad
cy+yc=xa,\quad dy+yd=xb,\\
x^2+2\Lam_4y^2=\mu(1-a^2),\quad xy+\Lam_4yx=\frac{1}{2}(\Lam_4-\Lam_2)(\Lam_2+1)\mu ca,\quad x^4=0,
\end{align*}
the coalgebra structure is given by
\begin{align*}
\Delta(a)&=a\otimes a+b\otimes c,\quad \Delta(b)=a\otimes b+b\otimes d,\\
\Delta(c)&=c\otimes a+d\otimes c,\quad \Delta(d)=d\otimes d+c\otimes b,\\
\Delta(x)&=x\otimes 1+a^{j}\otimes x-(\Lam_2+\Lam_4)ba^{j-1}\otimes y,\\
\Delta(y)&=y\otimes 1+da^{j-1}\otimes y+\frac{1}{2}\xi(\Lam_2-\Lam_4)ca^{j-1}\otimes x.
\end{align*}
\end{defi}
\begin{rmk}
It is clear that $\Lambda^5(0)\cong \BN(V_{i,j,k,\iota})\sharp H$ for $(i,j,k,\iota)\in\Lambda^5$. Moreover, $\LA\cong T(V_{i,j,k,\iota})\sharp H/J^5$ where $J^5$ is the ideal generated by the elements the last row of the equations.
\end{rmk}
\begin{defi}
For $\mu\in\K$ and $(i,j,k,\iota)\in \Lambda^6$. Let $\LAA$ be the algebra generated by $x$, $y$, $a$, $b$, $c$, $d$ satisfying the relations
\begin{align*}
a^4=1,\quad b^2=0,\quad c^2=0, \quad d^4=1,\quad a^2d^2=1,\quad ad=da,\quad bc=0=cb,\\
ab=\xi ba,\quad ac=\xi ca,\quad bd=\xi db,\quad cd=\xi dc,\quad bd=ca,\quad
ba=cd,\\
ax=-\xi xa,\quad bx=-\xi xb,\quad cx=\Lam_1 xc,\quad dx=\Lam_1 xd,\\
ay+ya=\Lam_4xc,\quad by+yb=\Lam_4xd,\quad
cy+\Lam_4yc=xa,\quad dy+\Lam_4yd=xb,\\
x^2+2\Lam_4y^2=\mu(1-a^2),\quad xy+\Lam_4yx=\Lam_4\mu ca,\quad x^4=0,
\end{align*}
the coalgebra structure is given by
\begin{align*}
\Delta(a)&=a\otimes a+b\otimes c,\quad \Delta(b)=a\otimes b+b\otimes d,\\
\Delta(c)&=c\otimes a+d\otimes c,\quad \Delta(d)=d\otimes d+c\otimes b,\\
\Delta(x)&=x\otimes 1+da^{j-1}\otimes x+(\Lam_2\Lam_4-1)ca^{j-1}\otimes y,\\
\Delta(y)&=y\otimes 1+a^{j}\otimes y-\frac{1}{2}(\Lam_2\Lam_4+1)ba^{j-1}\otimes x.
\end{align*}
\end{defi}

\begin{rmk}
It is clear that $\Lambda^6(0)\cong \BN(V_{i,j,k,\iota})\sharp H$ for $(i,j,k,\iota)\in\Lambda^6$. Moreover, $\LAA\cong T(V_{i,j,k,\iota})\sharp H/J^6$ where $J^6$ is the ideal generated by the elements the last row of the equations.
\end{rmk}
In the following Lemma, we show that $\LA$ and $\LAA$ are finite dimensional Hopf algebras over $H$.
\begin{lem}\label{lemLALAAoverH}
For $\mu\in\K$, $\ell=0,1$ and $(i,j,k,\iota)\in \Lambda^{\ell}$, $\Lambda^{\ell}(\mu)$ is finite-dimensional Hopf algebra over $H$.
\end{lem}
\begin{proof}
We prove the assertion for $\LA$, being the proof for $\LAA$ completely analogous.
Let $\Lambda_0$ be the subalgebra of $\LA$ for some $(i,j,k,\iota)\in\Lambda^5$ generated by the subcoalgebra $C=\K\{a,b,c,d\}$. We claim that $\Lambda_0\cong H$. Indeed, consider the Hopf algebra map $\psi:H\mapsto \LA$ given by the composition $H\hookrightarrow T(V_{i,j,k,\iota})\sharp H\twoheadrightarrow \LA\cong T(V_{i,j,k,\iota})\sharp H/J^5$. It is clear that $\psi(C)\cong C$ as coalgebras and $\psi(H)\cong \Lambda_0$ as Hopf algebras. From Remark $\ref{rmkHindependent}$ the elements $a$, $b$, $c$, $d$ are linearly-independent in $H$, thus they are also linearly-independent in $\LA$ which implies that $\dim\psi(H)> 4$. Then $\dim\psi(H)=8$ or $16$ by the Nichols-Zoeller Theorem. If $\dim\psi(H)=8$, then $\psi(H)$ must be the unique Hopf algebra of dimension $8$ without Chevalley property and it is impossible by the relations in \cite[Proposition\,2.1]{GG16}. Hence $\dim\psi(H)=16$ and $\psi(H)\cong H$ which implies that the claim follows.

Let $\Lambda_1=H\{x,y\}$,  $\Lambda_2=\Lambda_1+H\{x^2,xy\}$, $\Lambda_3=\Lambda_2+H\{x^3,x^2y\}$ and $\Lambda_4=\Lambda_3+H\{x^3y\}$. After a direct computation, one can show that $\{\Lambda_{\ell}\}_{\ell=0}^4$ is a coalgebra filtration of $\LA$. Hence $(\LA)_0\subseteq H$ and $(\LA)_{[0]}\cong H$ i.e., $\LA$ is a Hopf algebra over $H$. Moreover, $\LA$ is a finite-dimensional Hopf algebra which is free as $H-$module with a set of generators $\{1,x,y,x^2,xy,x^3,x^2y,x^3y\}$ and $\dim \LA\leq 8\dim H=128$.
\end{proof}

Now we show that the algebras $\LA$ and $\LAA$ are liftings of of the bosonizations $\BN(V_{i,j,k,\iota})\sharp H$ for $(i,j,k,\iota)\in \Lambda^5$ and $\Lambda^6$. The proof is much similar with that in \cite[Lemma\;5.9]{GG16}.
\begin{lem}
For $\ell=5,6$, $\text{gr}\,\Lambda^{\ell}(\mu)\cong \BN(V_{i,j,k,\iota})\sharp H$ where $(i,j,k,\iota)\in\Lambda^{\ell}$.
\end{lem}

\begin{pro}\label{proLiftingnon5}
Let $A$ be a finite-dimensional Hopf algebra over $H$ such that the infinitesimal braiding $V$ is isomorphic either to $V_{i,j,k,\iota}$ where $(i,j,k,\iota)\in \Lambda^5$. Then $A\cong \LA$ where $(i,j,k,\iota)\in \Lambda^5$.
\end{pro}
\begin{proof}
Note that $\text{gr}\; A\cong \BN(V_{i,j,k,\iota})\sharp H$, where $(i,j,k,\iota)\in \Lambda^5$. For convenience, let $W:=V_{i,j,k,\iota}$ where $(i,j,k,\iota)\in \Lambda^5$.
In such a case, the bosonization $\BN(W)\sharp H$ is generated by $x, y, a, b, c, d$ satisfying the relations
\begin{align*}
a^4=1,\quad b^2=0,\quad c^2=0, \quad d^4=1,\quad a^2d^2=1,\quad ad=da,\quad bc=0=cb,\\
ab=\xi ba,\quad ac=\xi ca,\quad bd=\xi db,\quad cd=\xi dc,\quad bd=ca,\quad
ba=cd.\\
ax=-\Lam_4\xi xa,\quad bx=-\Lam_4\xi xb,\quad cx=\xi xc,\quad dx=\xi xd,\\
ay+\Lam_4 ya=\Lam_4xc,\quad by+\Lam_4 yb=\Lam_4xd,\quad
cy+yc=xa,\quad dy+yd=xb,\\
x^2+2\Lam_4y^2=0,\quad x^4=0, \quad xy+\Lam_4yx=0.
\end{align*}
The coalgebra structure is given by
\begin{align*}
\Delta(a)&=a\otimes a+b\otimes c,\quad \Delta(b)=a\otimes b+b\otimes d,\\
\Delta(c)&=c\otimes a+d\otimes c,\quad \Delta(d)=d\otimes d+c\otimes b,\\
\Delta(x)&=x\otimes 1+a^{j}\otimes x-(\Lam_2+\Lam_4)ba^{j-1}\otimes y,\\
\Delta(y)&=y\otimes 1+da^{j-1}\otimes y+\frac{1}{2}\xi(\Lam_2-\Lam_4)ca^{j-1}\otimes x.
\end{align*}
Assume that $A$ is a finite-dimensional Hopf algebra such that $\text{gr}\,A\cong \BN(W)\sharp H$. After a direct computation, we have that
\begin{align*}
\Delta(x^2+2\Lam_4y^2)&=(x^2+2\Lam_4y^2)\otimes 1+a^2\otimes (x^2+2\Lam_4y^2),\\
\Delta(xy+\Lam_4yx)&=(xy+\Lam_4yx)\otimes 1+\frac{1}{2}(\Lam_2-\Lam_4)(\Lam_2+1)ca\otimes (v_1^2+2\Lam_4v_2^2)\\&\quad
+da\otimes(xy+\Lam_4yx).
\end{align*}
From the first equation, we have that
\begin{align*}
x^2+2\Lam_4y^2\in\Pp_{1,a^2}(\BN(W)\sharp H)=\Pp_{1,a^2}(H)=\K\{1-a^2\},
\end{align*}
which implies that $x^2+2\Lam_4y^2=\mu(1-a^2)$ for some $\mu\in\K$. Then from the second equation, we get that
\begin{align*}
\Delta(xy+\Lam_4yx+\omega ca)
&=(xy+\Lam_4yx)\otimes 1+da\otimes(xy+\Lam_4yx)+\omega ca\otimes (1-a^2)\\&\quad
+\omega ca\otimes a^2+\omega da\otimes ca\\
&=(xy+\Lam_4yx+\omega ca)\otimes 1+da\otimes (xy+\Lam_4yx+\omega  ca),
\end{align*}
where $\omega=\frac{1}{2}(\Lam_2-\Lam_4)(\Lam_2+1)\mu $.
Thus there exists $\nu\in\K$ such that
\begin{align*}
xy+\Lam_4yx+\frac{1}{2}(\Lam_2-\Lam_4)(\Lam_2+1)\mu ca=\nu(1-da).
\end{align*}
Since
\begin{align*}
\nu(1-da)c&=\nu c(1-da),\\
c(xy+\Lam_4yx)&=cxy+\Lam_4yx=\xi xcy+\Lam_4cyx\\
&=\xi x(xa-yc)+\Lam_4(xa-yc)x=-\xi(xy+\Lam_4yx)c,
\end{align*}
we get that $\nu=0$ and whence $xy+\Lam_4yx=\frac{1}{2}(\Lam_4-\Lam_2)(\Lam_2+1)\mu ca$. Finally, for $x^4$, we have that
\begin{align*}
\Delta(x^4)=\Delta(x)^4=x^4\otimes 1+1\otimes x^4,
\end{align*}
which implies that the relation $x^4=0$ must hold in $A$.
\end{proof}

\begin{pro}\label{proLiftingnon6}
Let $A$ be a finite-dimensional Hopf algebra over $H$ such that the infinitesimal braiding $V$ is isomorphic either to $V_{i,j,k,\iota}$ where $(i,j,k,\iota)\in \Lambda^6$. Then $A\cong \LAA$ where $(i,j,k,\iota)\in \Lambda^6$.
\end{pro}
\begin{proof}
Note that $\text{gr}\, A\cong \BN(V_{i,j,k,\iota})\sharp H$, where $(i,j,k,\iota)\in \Lambda^6$. For convenience, let $W:=V_{i,j,k,\iota}$ where $(i,j,k,\iota)\in \Lambda^6$.
In such a case, the bosonization $\BN(W)\sharp H$ is generated by $x, y, a, b, c, d$ satisfying the relations
\begin{align*}
a^4=1,\quad b^2=0,\quad c^2=0, \quad d^4=1,\quad a^2d^2=1,\quad ad=da,\quad bc=0=cb,\\
ab=\xi ba,\quad ac=\xi ca,\quad bd=\xi db,\quad cd=\xi dc,\quad bd=ca,\quad
ba=cd.\\
ax=-\xi xa,\quad bx=-\xi xb,\quad cx=\Lam_1 xc,\quad dx=\Lam_1 xd,\\
ay+ya=\Lam_4xc,\quad by+yb=\Lam_4xd,\quad
cy+\Lam_4yc=xa,\quad dy+\Lam_4yd=xb,\\
x^2+2\Lam_4y^2=0,\quad x^4=0, \quad xy+\Lam_4yx=0.
\end{align*}
the coalgebra structure is given by
\begin{align*}
\Delta(a)&=a\otimes a+b\otimes c,\quad \Delta(b)=a\otimes b+b\otimes d,\\
\Delta(c)&=c\otimes a+d\otimes c,\quad \Delta(d)=d\otimes d+c\otimes b,\\
\Delta(x)&=x\otimes 1+da^{j-1}\otimes x+(\Lam_2\Lam_4-1)ca^{j-1}\otimes y,\\
\Delta(y)&=y\otimes 1+a^{j}\otimes y-\frac{1}{2}(\Lam_2\Lam_4+1)ba^{j-1}\otimes x.
\end{align*}
Assume that $A$ is a finite-dimensional Hopf algebra such that $\text{gr}\,A\cong \BN(W)\sharp H$. After a direct computation, we have that
\begin{gather*}
\Delta(x^2+2\Lam_4y^2)=(x^2+2\Lam_4y^2)\otimes 1+a^2\otimes (x^2+2\Lam_4y^2),\\
\Delta(xy+\Lam_4yx)=(xy+\Lam_4yx)\otimes 1+da\otimes(xy+\Lam_4yx)-\Lam_4ca\otimes (v_1^2+2\Lam_4v_2^2).
\end{gather*}
From the first equation, we have that
\begin{align*}
x^2+2\Lam_4y^2\in\Pp_{1,a^2}(\BN(W)\sharp H)=\Pp_{1,a^2}(H)=\K\{1-a^2\},
\end{align*}
which implies that $x^2+2\Lam_4y^2=\mu(1-a^2)$ for some $\mu\in\K$. Then from the second equation, we get that
\begin{align*}
\lefteqn{\Delta(xy+\Lam_4yx-\Lam_4\mu ca)}\hspace*{2mm}\\
&=(xy+\Lam_4yx)\otimes 1+da\otimes(xy+\Lam_4yx)-\Lam_4ca\otimes \mu(1-a^2)\\&\quad
-\mu\Lam_4ca\otimes a^2-\mu\Lam_4da\otimes ca\\&
=(xy+\Lam_4yx-\Lam_4\mu ca)\otimes 1+da\otimes (xy+\Lam_4yx-\Lam_4\mu ca),
\end{align*}
which implies that there exists $\nu\in\K$ such that
\begin{align*}
xy+\Lam_4yx-\Lam_4\mu ca=\nu(1-da).
\end{align*}
After a direct computation, we have that $c(xy+\Lam_4yx)=-\xi(xy+\Lam_4)c$ and $c(1-da)=(1-da)c$ which implies  that $\nu=0$ and whence $xy+\Lam_4yx=\Lam_4\mu ca$. Finally, for $x^4$ we have that
\begin{align*}
\Delta(x^4)=\Delta(x)^4=x^4\otimes 1+1\otimes x^4,
\end{align*}
which implies that the relation $x^4=0$ holds in $A$.
\end{proof}

Finally, we have the classification of finite-dimensional Hopf algebras over $H$ such that their infinitesimal braidings are simple objects in $\HYD$.

\begin{proofthmb}
Since $A$ is a finite-dimensional Hopf algebra over $H$, that is, $A_{[0]}\cong H$, by Lemma $\ref{lemNicholsgeneratedbyone}$, we get that $\text{gr}\,A=\BN(V)\sharp H$. If $V$ is isomorphic either to $\K_{\chi_{i,j,k}}$ with $(i,j,k)\in\Lambda^0$, or $V_{i,j,k,\iota}$ with $(i,j,k,\iota)\in\Lambda^3\cup\Lambda^4$, then $A$ is isomorphic either to $\bigwedge\K_{\chi_{i,j,k}}\sharp H$ with $(i,j,k)\in\Lambda^0$, or $\BN(V_{i,j,k,\iota})\sharp H$ with $(i,j,k,\iota)\in \Lambda^3\cup\Lambda^4$ by Proposition $\ref{proLiftingnon0}$, $\ref{proLiftingnon3}$ and $\ref{proLiftingnon4}$.
If $V$ is isomorphic either to $V_{i,j,k,\iota}$ with $(i,j,k,\iota)\in\Lambda^5\cup\Lambda^6$, then $A$ is isomorphic either to $\LA$ or $\LAA$, with $(i,j,k,\iota)\in \Lambda^3\cup\Lambda^4$ by Propositions $\ref{proLiftingnon5}$ and $\ref{proLiftingnon6}$. And these Hopf algebras are pairwise non-isomorphic  since their infinitesimal braidings are pairwise non-isomorphic as Yetter-Drinfeld modules over $H$.
\end{proofthmb}

\vskip10pt \centerline{\bf ACKNOWLEDGMENT}

\vskip10pt The paper is supported by the NSFC (Grant No. 11271131).
The authors are grateful to Andruskiewitsch for his kind comments.


\end{document}